\documentclass[a4paper,10.5pt]{article}
\usepackage{amsmath,amssymb,amsthm}

\usepackage{color}
\theoremstyle{definition}
 \newtheorem{dfn}{Definition}[section]
 \newtheorem{remark}[dfn]{Remark}

\theoremstyle{plain}
 \newtheorem{thm}[dfn]{Theorem}
   
 \newtheorem{lem}[dfn]{Lemma}

\numberwithin{equation}{section}

\newcommand{\bA}{{\mathbb A}}

\newcommand{\bD}{{\bold D}}

\newcommand{\bG}{{\bold G}}
\newcommand{\bH}{{\mathbb H}}
\newcommand{\bI}{{\mathbb I}}

\newcommand{\bU}{{\bold U}}
\newcommand{\bL}{{\mathbb L}}
\newcommand{\bM}{{M}}

\newcommand{\DV}{{\rm Div}\,}
\newcommand{\dv}{\, {\rm div}\,}

\newcommand{\BR}{{\Bbb R}}
\newcommand{\BC}{{\Bbb C}}

\newcommand{\BN}{{\Bbb N}}
\newcommand{\BG}{{\Bbb G}}

\newcommand{\BH}{{\Bbb H}}
\newcommand{\BI}{{\Bbb I}}

\newcommand{\BP}{{\Bbb P}}
\newcommand{\BQ}{{\Bbb Q}}

\newcommand{\BS}{{\Bbb S}}

\newcommand{\CA}{{\mathcal A}}
\newcommand{\CB}{{\mathcal B}}

\newcommand{\CD}{{\mathcal D}}

\newcommand{\CF}{{\mathcal F}}
\newcommand{\CI}{{\mathcal I}}

\newcommand{\CL}{{\mathcal L}}

\newcommand{\CN}{{\mathcal N}}

\newcommand{\CR}{{\mathcal R}}
\newcommand{\CS}{{\mathcal S}}
\newcommand{\CT}{{\mathcal T}}

\newcommand{\bff}{{\bold f}}
\newcommand{\bv}{{\bold v}}
\newcommand{\bu}{{\bold u}}

\newcommand{\bg}{{\mathbb G}}

\newcommand{\bQ}{{\mathbb Q}}
\newcommand{\bP}{{\mathbb P}}
\newcommand{\bS}{{\mathbb S}}
\newcommand{\bF}{{\bold F}}

\newcommand{\pd}{\partial}

\newcommand{\R}{\mathbb{R}}
\newcommand{\N}{\mathbb{N}}
\newcommand{\C}{\mathbb{C}}

\newcommand{\vp}{\varphi}

\newcommand{\fp}{{\frak p}}
\newcommand{\tr}{\mathrm{tr}}

\title{Global well posedness for a Q-tensor model of nematic liquid crystals}
\author{
Miho MURATA
\thanks{Department of Mathematical and System Engineering,
Faculty of Engineering,
Shizuoka University, 
\endgraf
3-5-1 Johoku, Naka-ku, Hamamatsu-shi, Shizuoka,
432-8561, Japan.
\endgraf
e-mail address: murata.miho@shizuoka.ac.jp
\endgraf
Partially supported by JSPS
Grant-in-aid for Young Scientists 21K13819
}\enskip and \enskip
Yoshihiro SHIBATA
\thanks{Department of Mathematics,  
Waseda University, \endgraf
Ohkubo 3-4-1, Shinjuku-ku, Tokyo 169-8555, Japan. \endgraf
e-mail address: yshibata@waseda.jp
\endgraf
Adjunct faculty member in the Department of Mechanical Engineering and
Materials Science, University of Pittsburgh.
\endgraf
Partially support by JSPS Grant-in-aid for Scientific Research (A) 17H0109
and 
Top Global University Project}
}
\date{}

\begin{document}
\maketitle

\begin{abstract}
In this paper, we prove the global well posedness 
and the decay estimates for a $\bQ$-tensor model of nematic liquid crystals
in $\BR^N$, $N \geq 3$.
This system is coupled system by the Navier-Stokes equations with a parabolic-type equation
describing the evolution of the director fields $\bQ$. 
The proof is based on the maximal $L_p$ -$L_q$ regularity and the $L_p$ -$L_q$ decay estimates
to the linearized problem. 
\end{abstract}

\section{Introduction}

We consider the model for  
a viscous incompressible liquid crystal flow
proposed by Beris and Edwards \cite{BE}.
In this model, the molecular orientation of the liquid crystal
is described by a tensor $\bQ=\bQ(x, t)$. 
More precisely, we consider the following system in the $N$ dimensional 
Euclidean space $\BR^N$, $N \geq 3$. 
\begin{equation}\label{nonlinear0}
\left\{
\begin{aligned}
&\pd_t \bu + (\bu \cdot \nabla) \bu + \nabla \fp
= \Delta \bu +\DV (\tau(\bQ)+\sigma(\bQ)), 
\enskip \dv \bu=0
& \quad&\text{in $\R^N$ for $t \in \BR$}, \\
&\pd_t \bQ + (\bu \cdot \nabla) \bQ - \bS(\nabla \bu, \bQ)
=\bH & \quad&\text{in $\R^N$ for $t \in \BR$}, \\
&(\bu, \bQ) = (\bu_0, \bQ_0)& \quad&\text{in $\R^N$}.
\end{aligned}\right.
\end{equation}
Here, $\pd_t = \pd/\pd t$, $t$ is the time variable, 
$\bu(x, t) = (u_1(x, t), \ldots, u_N(x, t))^T$ is the fluid velocity,
where $\bM^T$ denotes the transposed $\bM$, 
and $\fp=\fp(x, t)$ is the pressure.
For vector of functions $\bv$,
we set $\dv \bv = \sum_{j=1}^N\pd_j v_j$,
and also for $N\times N$ matrix field $\bA$ with $(j,k)^{\rm th}$ components $A_{jk}$, 
the quantity $\DV \bA$ is an 
$N$-vector with $j^{\rm th}$ component $\sum_{k=1}^N\pd_kA_{jk}$,
where $\pd_k=\pd/\pd x_k$.
The tensors $\bS(\nabla \bu, \bQ)$, $\tau(\bQ)$, and $\sigma(\bQ)$ are
\begin{align*}
\bS(\nabla \bu, \bQ)&=
(\xi \bD(\bu)+ W(\bu)) \left( \bQ + \frac{1}{N} \bI \right)
+\left( \bQ + \frac{1}{N} \bI \right)(\xi \bD(\bu)- W(\bu))
-2\xi\left( \bQ + \frac{1}{N} \bI \right) \bQ : \nabla \bu,\\
\tau(\bQ) 
&=2\xi \bH: \bQ\left( \bQ + \frac{1}{N} \bI \right)
-\xi\left[\bH\left( \bQ + \frac{1}{N} \bI \right) + \left( \bQ + \frac{1}{N} \bI \right) \bH\right]
- \nabla \bQ \odot \nabla \bQ,
\\
\sigma(\bQ)
&=\bQ \bH - \bH \bQ,
\end{align*}
where 
$\bD(\bu) = (\nabla \bu +(\nabla \bu)^T)/2$ and
$W(\bu) = (\nabla \bu -(\nabla \bu)^T)/2$
denote the symmetric and antisymmetric part of $\nabla \bu$,
respectively,
the $(i, j)$ component of $\nabla \bQ \odot \nabla \bQ$
is $\sum^N_{\alpha, \beta = 1} \pd_i Q_{\alpha\beta} \pd_j Q_{\alpha\beta}$,
$\bI$ is the $N\times N$ identity matrix.
$\bS(\nabla \bu, \bQ)$ describes how the flow gradient rotates and stretches 
the order parameter $\bQ$.
A scalar parameter $\xi\in \R$
denotes the ratio between the tumbling 
and the aligning effects that a shear flow would exert over the directors.
$\bH$ is given by 
\[
\bH = \Delta \bQ - \CL[\pd F(\bQ)],
\]
where 
$F(\bQ)$ denotes the bulk energy of Landau-de Gennes type:
\[
F(\bQ)=\frac{a}{2} \tr(\bQ^2) - \frac{b}{3} \tr(\bQ^3) + \frac{c}{4}(\tr (\bQ^2))^2
\]
with material-dependent and temperature-dependent constants $a, b, c \in \R$.
$\CL[\bA]$ denotes the projection 
onto the space of traceless matrices, namely,
\[
\CL[\bA] = \bA -\frac{1}{N}\tr [\bA]\bI.
\]
Note that $\bH$ is deriven from the variational derivative of free energy:
\[
\CF(\bQ) = \int_{\R^N} \left(\frac{L}{2}|\nabla \bQ|^2 + F(\bQ)\right)\,dx,
\]
which is the one-constant approximation of 
the general Oseen-Frank energy (cf.\cite{BM}),
where the gradient term corresponds to the elastic part of the free energy and 
$L>0$ is the elastic constant.
In our model, we set $L=1$ for simplicity. 
Here and hereafter, we assume that
\[
a, c>0.
\]
The assumption $c>0$ is deriven from a modeling point of view 
to guarantee that the free energy $\CF(\bQ)$ is bounded from below (cf. \cite{M, MZ}). 
On the other hand, at least in the case $a>0$, 
we can prove the existence of global strong solutions
because an eigenvalue of the linear operator corresponding to \eqref{nonlinear0}
does not appear on the positive real axis.

The molecules of nematic liquid crystals flow as in a liquid phase,
but they have long-range orientation order. 
Such a continuum theory for the hydrodynamics of nematic liquid crystals 
was first studied by Ericksen \cite{E1} in 1962. 
Based on his idea, Ericksen and Leslie proposed 
the so called Ericksen-Leslie model \cite{E2, L1, L2} in 1960s, 
in which the evolution of the unit director field is coupled with 
an evolution equation for the underlying flow field 
which is given by the Navier-Stokes equation with an additional forcing term.
However, the Ericksen-Leslie model could not describe the biaxial nematic liquid crystals.
In order to treat the biaxial nematic liquid crystals,
P.G. de Gennes \cite{DG} introduced an $N \times N$ symmetric, traceless matrix $\bQ$ as the new parameter,
which is called $\bQ$-tensor.
After that, Beris and Edwards \cite{BE} proposed a model, 
in which the director field was replaced by a $\bQ$-tensor.

Mathematically the Beris-Edwards model has been studied by many authors recent years.
Concerning weak solutions, 
Paicu and Zarnescu \cite{PZ2} obtained the first result for the simplified model with $\xi= 0$. 
They proved the existence of global weak solutions
in $\R^N$ with $N=2, 3$ as well as weak-strong uniqueness for $N=2$.
Here, $\xi=0$ means that the molecules are such that they only tumble in a shear flow, but are not aligned by such a flow.
In \cite{PZ1}, they proved the same results as in \cite{PZ2} for the case that $\xi$ is sufficiently small. 
An improved result on weak solutions in $\R^2$ was established in \cite{D}.
Huang and Ding \cite{HD} proved the existence of global weak solutions with a more
general energy functional and $\xi=0$ in $\R^3$.
 
On the other hand, concerning strong solutions, 
Abels, Dolzmann, and Liu \cite{Ab1} showed the existence of a strong local solution and 
global weak solutions with higher regularity 
in time in the case of inhomogeneous mixed Dirichlet/Neumann boundary conditions in a bounded domain
without any smallness assumption on the parameter $\xi$. 
Liu and Wang \cite{LW} improved the spatial regularity of solutions obtained in \cite{Ab1}
and generalized their result to the case of anisotropic elastic energy.
Abels, Dolzmann, and Liu \cite{Ab2} also proved the local well posedness with Dirichlet boundary condition 
for the classical Beris-Edwards model, 
which means that fluid viscosity depends on the $\bQ$-tensor, 
but for the case $\xi=0$ only. 
Cavaterra et al. \cite{C} showed the global well posedness in the two dimensional periodic case
without any smallness assumption on the parameter $\xi$. 
Xiao \cite{X} proved the global well posedness for the simplified model with $\xi= 0$ in a bounded domain.
He constructed a solution in the maximal $L_p$-$L_q$ regularity class.

In this paper, we treat the full model described by \eqref{nonlinear0}. Since 
the linear part of 
$\xi[\BH(\BQ+\frac1N\BI)+ (\BQ+\frac1N\BI)\BH$ is $\frac{2\xi}{N}\Delta \BQ$, 
we separate problem \eqref{nonlinear0} into linear part and nonlinear part 
as follows: 
\begin{equation}\label{nonlinear}
\left\{
\begin{aligned}
&\pd_t \bu -\Delta \bu + \nabla \fp +\beta \DV \Delta \bQ = \bff(\bu, \bQ), 
\enskip \dv \bu=0
& \quad&\text{in $\R^N$ for $t \in \BR$}, \\
&\pd_t \bQ - \beta \bD(\bu) -\Delta \bQ + a\left(\bQ - \frac{1}{N} \tr \bQ \bI \right)
=\bg(\bu, \bQ) & \quad&\text{in $\R^N$ for $t \in \BR$}, \\
&(\bu, \bQ) = (\bu_0, \bQ_0)& \quad&\text{in $\R^N$},
\end{aligned}\right.
\end{equation}
where 
\begin{align*}
\beta
&= 2\xi/N,\\
\bff(\bu, \bQ)
&= -\bu \cdot \nabla \bu + 
\DV[2\xi \bH: \bQ(\bQ+\bI/N) - (\xi + 1) \bH\bQ +
 (1-\xi) \bQ\bH -\nabla \bQ \odot \nabla \bQ],\\
\bg(\bu, \bQ)
&= -(\bu \cdot \nabla)\bQ + \xi(\bD(\bu) \bQ + \bQ \bD(\bu))
+W(\bu) \bQ - \bQ W(\bu)\\
&\enskip -2\xi (\bQ +\bI/N) \bQ : \nabla \bu - \CL[\pd F'(\bQ)],\\
F'(\bQ)
&= (b/3)\tr(\bQ^3) + (c/4)(\tr(\bQ^2))^2.
\end{align*}
Recently, Schonbek and the second author in \cite{SS1} considered 
the Beris-Edwards system removed 
$\Delta \bQ$, that is 
$\bH(\BQ+\frac1N\BI)+(\BQ+\frac1N \BI)\BH$ is replaced by 
$\bH(\BQ+\frac1N\BI)-(\BQ+\frac1N \BI)\BH$
in the tensor $\tau(\BQ)$ and proved the global well posedness
for small initial data in the following solution class:
\begin{equation}\label{max.reg.1}\begin{aligned}
\bu &\in \bigcap_{q=q_1, q_2} 
H^1_p((0, T), L_q(\BR^N)^N) \cap L_p((0, T), H^2_q(\BR^N)^N),
\\
 \BQ &\in \bigcap_{q=q_1, q_2} H^1_p((0, T), H^1_q(\BR^N)^{N^2}) \cap 
L_p((0, T), H^3_q(\BR^N)^{N^2}) 
\end{aligned}\end{equation}
with certain $p$, $q_1$, and $q_2$. 
 In \cite{SS1}, the linear part of the first equation in \eqref{nonlinear}
is $\pd_t \bu -\Delta \bu + \nabla \fp$, and so 
$\bu$ part and 
$\BQ$ part  of linearized
equations are essentially separated.  On the other hand,   
in the present paper,  equations for $\bu$ and $\BQ$ are 
completely coupled  by third order term: $\beta \DV \Delta \bQ$.
This is a big difference between \cite{SS1} and the present paper. 

In this paper, the global well posedness for small initial data shall be
proved  in the solution class \eqref{max.reg.1} with the help of  
the maximal regularity theory and $L_p$-$L_q$ decay properties of 
solutions to linearized equations. The spirit to use 
both of them is the same as in \cite{SS1}, 
but the idea how to use them is different, and we think that our approach
here  gives a general framework to prove the global well posedness
for small initial data of quasilinear parabolic equations in unbounded 
domains.  
To explain our idea  more precisely, we write equations as 
 $\pd_t u - Au = f$ and $u|_{t=0}=u_0$ symbolically,
where $u = (\bu, \BQ)$, $f$ is the corresponding nonlinear term and $A$ is a closed linear operator with domain $D(A)
=(W^2_{q_1} \cap W^2_{q_2})\times (W^3_{q_1} \cap W^3_{q_2})$, where $q_1$ and $q_2$ are
some exponents such that $2 < q_1< N < q_2$
chosen precisely in the statement of our main result below.  We use the maximal 
$L_p$-$D(A)$ regularity for the time shifted equations $\pd_t u + \lambda_1u - Au = f$ and $u|_{t=0}=0$, and 
 $L_p$-$L_q$ decay estimates of continuous analytic semigroup $\{e^{At}\}_{t\geq 0}$ associated with the operator $A$
for $t>1$ to prove the decay properties of solutions 
to the compensation equations: $\pd_t v-Av= \lambda_1u$ and $v|_{t=0}=u_0$.
The Duhamel's principle implies that $v = e^{At}u_0 +\lambda_1 \int^t_0 e^{A(t-s)}u(s)\,ds$. 
To estimate $\int^{t-1}_0 e^{A(t-s)}u(s)\,ds$ 
we use the decay properties of $e^{A(t-s)}$ for $t-s > 1$,
and to estimate $\int^t_{t-1}e^{A(t-s)}u(s)\,ds$ we use a standard estimate: 
$\|e^{A(t-s)}u(s)\|_{D(A)} \leq C\|u(s)\|_{D(A)}$ for $0 < t-s < 1$. 
For the later part,  what $u(t) \in D(A)$ for $t>0$ is a key observation.
The key issues of this paper are the maximal $L_p$-$D(A)$ 
maximal regularity of the operator $A$ and 
 the $L_p$-$L_q$ decay properties of $\{e^{At}\}_{t \geq 0}$, both of which are new results and  shall be proved respective 
in Sect. 2 and in Sect. 3 below.  

Before stating the main result of this paper,  
we summarize several symbols and functional spaces used 
throughout the paper.
$\BN$, $\BR$ and $\BC$ denote the sets of 
all natural numbers, real numbers and complex numbers, respectively. 
We set $\BN_0=\BN \cup \{0\}$ and $\BR_+ = (0, \infty)$. 
Let $q'$ be the dual exponent of $q$
defined by $q' = q/(q-1)$
for $1 < q < \infty$. 
For any multi-index $\alpha = (\alpha_1, \ldots, \alpha_N) 
\in \BN_0^N$, we write $|\alpha|=\alpha_1+\cdots+\alpha_N$ 
and $\pd_x^\alpha=\pd_1^{\alpha_1} \cdots \pd_N^{\alpha_N}$ 
with $x = (x_1, \ldots, x_N)$. 
For scalar function $f$, $N$-vector of functions ${\bold g}$, 
and $N \times N$ matrix fields $\bL$, we set
\begin{gather*}
\nabla^k f = (\pd_x^\alpha f \mid |\alpha|=k),
\enskip \nabla^k {\bold g} = (\pd_x^\alpha g_j \mid |\alpha|=k, \enskip j = 1,\ldots, N),\\
\nabla^k \bL = (\pd_x^\alpha L_{j \ell} \mid |\alpha|=k, \enskip j, \ell = 1,\ldots, N).
\end{gather*} 
For Banach spaces $X$ and $Y$, $\CL(X,Y)$ denotes the set of 
all bounded linear operators from $X$ into $Y$,
$\CL(X)$ is the abbreviation of $\CL(X, X)$, and 
$\rm{Hol}\,(U, \CL(X,Y))$ 
 the set of all $\CL(X,Y)$ valued holomorphic 
functions defined on a domain $U$ in $\BC$. 
For any $1 \leq p, q \leq \infty$,
$L_q(\BR^N)$, $W_q^m(\BR^N)$ and $B^s_{q, p}(\BR^N)$ 
denote the usual Lebesgue space, Sobolev space and 
Besov space, 
while $\|\cdot\|_{L_q(\BR^N)}$, $\|\cdot\|_{W_q^m(\BR^N)}$ and 
$\|\cdot\|_{B^s_{q,p}(\BR^N)}$ 
denote their norms, respectively. We set $W^0_q(\BR^N) = L_q(\BR^N)$
and $W^s_q(\BR^N) = B^s_{q,q}(\BR^N)$. 
$C^\infty(\BR^N)$ denotes the set of all $C^\infty$ functions defined on $\BR^N$. 
$L_p((a, b), X)$ and $W_p^m((a, b), X)$ 
denote the standard Lebesgue space and Sobolev space of 
$X$-valued functions defined on an interval $(a,b)$, respectively.
The $d$-product space of $X$ is defined by 
$X^d=\{\bff=(f_1, \ldots, f_d) \mid f_i \in X \, (i=1,\ldots,d)\}$,
while its norm is denoted by 
$\|\cdot\|_X$ instead of $\|\cdot\|_{X^d}$ for the sake of 
simplicity. 
We set 
\begin{gather*}
W_q^{m,\ell}(\BR^N)=\{(\bff,\bg) \mid  \bff \in W_q^m(\BR^N)^N,
\enskip \bg \in W_q^\ell(\BR^N)^{N^2} \}, \enskip 
\|(\bff, \bg)\|_{W^{m, \ell}_q(\BR^N)} = \|\bff\|_{W^m_q(\BR^N)}
+ \|\bg\|_{W^\ell_q(\BR^N)}.
\end{gather*}
Furthermore, we set
\begin{align*}
L_{p, \gamma}(\BR_+, X) & = \{f (t) \in L_{p, {\rm loc}}
(\BR_+, X) \mid e^{-\gamma t} f (t) \in L_p (\BR_+, X)\}, \\
W^1_{p, \gamma}(\BR_+, X) & = \{f(t) \in
 L_{p, \gamma}(\BR_+, X) \mid 
e^{-\gamma t} \pd_t^j f(t) \in L_p(\BR_+, X)
\enskip (j=0, 1)\}
\end{align*}
for $1 < p < \infty$ and $\gamma > 0$.
Let $\CF_x= \CF$ and $\CF^{-1}_\xi = \CF^{-1}$ 
denote the Fourier transform and 
the Fourier inverse transform, respectively, which are defined by 
 setting
$$\hat f (\xi)
= \CF_x[f](\xi) = \int_{\BR^N}e^{-ix\cdot\xi}f(x)\,dx, \quad
\CF^{-1}_\xi[g](x) = \frac{1}{(2\pi)^N}\int_{\BR^N}
e^{ix\cdot\xi}g(\xi)\,d\xi. 
$$
 The letter $C$ denotes generic constants and the constant 
$C_{a,b,\ldots}$ depends on $a,b,\ldots$. 
The values of constants $C$ and $C_{a,b,\ldots}$ 
may change from line to line. We use small boldface letters, e.g. $\bu$ to 
denote vector-valued functions and capital boldface letters, e.g. $\bH$
to denote matrix-valued functions, respectively. 
In order to state our main theorem, 
we introduce spaces
and several norms:
\begin{align}
J_q(\R^N)&=
\{\bu \in L_q(\R^N)^N \mid \dv \bu = 0 \text{ in $\R^N$}\}, \nonumber \\
D_{q, p}(\R^N)&=\{(\bu, \BQ) \mid \bu \in B^{2(1-1/p)}_{q, p}(\R^N)^N \cap J_q(\R^N)), \enskip 
\BQ \in B^{1+2(1-1/p)}_{q, p}(\R^N)^{N^2}\},  \nonumber \\
X_{p, q, t} &= \{(\bu, \bQ) 
 \mid \bu \in 
L_p((0, t), W^2_q(\BR^N)^N) \cap W^1_p((0, t), L_q(\BR^N)^N), \nonumber \\
& \bQ \in L_p((0, t), W^3_q(\BR^N)^{N^2}) \cap W^1_p((0, t), W^1_q(\BR^N)^{N^2})
 \}, 
\nonumber \\
\CN (\bu, \bQ) (T)
&=\sum_{q=q_1, q_2} 
\left(\|<t>^b (\bu, \bQ)\|_{L_\infty((0, T), W^{0, 1}_q(\R^N))}
+ \|<t>^b (\pd_t \bu, \pd_t \bQ)\|_{L_p((0, T), W^{0, 1}_q(\R^N))}\right)
 \nonumber \\
&\enskip +\|<t>^b (\nabla \bu, \nabla \bQ)\|_{L_p((0, T), W^{1, 2}_{q_1}(\R^N))}
+ \|<t>^b (\bu, \bQ)\|_{L_p((0, T), W^{2, 3}_{q_2}(\R^N))}, \label{N}
\end{align}
where $<t>=(1+t^2)^{1/2}$, $b$ is given in Theorem \ref{global} below.

The following theorem is  our main result of this paper.

\begin{thm}\label{global} Assume that $N\geq 3$.
Let $0 < T < \infty$ and $0 < \sigma < 1/2$.
Let $q_1$, $q_2$ and $p$ be numbers such that
\begin{equation}\label{condi:pq}
p= 2+\sigma,
\enskip
q_1=2+\sigma, 
\enskip
\left\{
\begin{aligned}
&q_2 \geq \frac{N(2+\sigma)}{N-(2+\sigma)}& &{\text if~} N=3, 4, \\
&q_2>N& &{\text if~} N\geq 5.
\end{aligned}
\right.
\end{equation}
Let $b=N/(2(2+\sigma))$.
Then, there exists a small number $\epsilon>0$ such that for any initial data 
$(\bu_0, \bQ_0) 
\in \bigcap^2_{i=1} D_{q_i, p} (\BR^N) \cap W^{0, 1}_{q_1/2}(\BR^N)$ with
\[
\CI :=
\sum^2_{i=1}\|(\bu_0, \bQ_0)\|_{D_{q_i, p} (\BR^N)}
+\|(\bu_0, \bQ_0)\|_{W^{0, 1}_{q_1/2}(\BR^N)} < \epsilon^2,
\]
 problem \eqref{nonlinear0} 
admits a unique solution $(\bu, \bQ)$ with
\[
(\bu, \bQ)\in X_{p, q_1, T} \cap X_{p, q_2, T}
\]
satisfying the estimate
\[
\CN(\bu, \bQ)(T)\leq \epsilon.
\]

\end{thm}

\begin{remark} \thetag1~$T>0$ is taken arbitrarily
 and $\epsilon$ is chosen independent of $T$,
therefore, Theorem \ref{global} 
yields the global well posedness for  
\eqref{nonlinear0}.\\
\thetag2~ 
$q_2$ satisfies $q_2 > N$ for any $N \geq 3$ because
$N(2+\sigma)/(N-(2+\sigma)) >N$ when $N=3, 4$.
By \eqref{condi:pq}, we can choose $q_2=15$ when $N=3$,
for instance.
\\
\thetag3~To get a priori estimates,
we need conditions 
$bp' > 1$ and $(N/(2(2+\sigma)) + 1/2 -b)p > 1$
for $0 < \sigma <1/2$.
We also need the condition $bp>1$ in order to estimate nonlinear term 
$<t>^b \bQ^2$ in $L_p((0, T), L_{q_1/2}(\R^N))$. 
Thus, we choose $b=N/(2(2+\sigma))$ and $p=2+\sigma$.
For details see Sect. \ref{proof of main} below. 
In the case $N = 2$, 
our argument does not work because
we have $p'b < 1$ if $N=2$. \\
\thetag4~If initial data $\bQ_0$ is a symmetric and traceless matrix,
 so is a solution $\bQ$, which can be proved 
by the uniqueness of  heat equations.
\end{remark}

This paper is organized as follows:
Sect. \ref{sec:mr}
proves the maximal $L_p$-$L_q$ regularity
by combining the existence of $\CR$-bounded solution operator families
to the resolvent problem and the Weis
operator valued Fourier multiplier theorem.
Sect. \ref{sec:decay} proves the $L_p$-$L_q$ decay estimates to the linearized problem.
Sect. \ref{proof of main} proves the main theorem with the help of 
the maximal $L_p$ -$L_q$ regularity and the $L_p$ -$L_q$ decay estimates.


\section{Maximal $L_p$-$L_q$ regularity}\label{sec:mr}
In this section, we show the maximal $L_p$-$L_q$ regularity 
for problem:
\begin{equation}\label{l0}
\left\{
\begin{aligned}
&\pd_t \bu -\Delta \bu + \nabla \fp +\beta \DV \Delta \bQ = \bff, 
\enskip \dv \bu=0
& \quad&\text{in $\R^N$ for $t>0$}, \\
&\pd_t \bQ - \beta \bD(\bu) -\Delta \bQ + a\left(\bQ - \frac{1}{N} \tr \bQ \bI \right)
=\bg & \quad&\text{in $\R^N$ for $t>0$}, \\
&(\bu, \bQ) = (\bu_0, \bQ_0)& \quad&\text{in $\R^N$}.
\end{aligned}\right.
\end{equation}
We now state the maximal $L_p$-$L_q$ regularity theorem.

\begin{thm}\label{thm:mr}
Let $1 < p, q < \infty$.
Then, there exists a constant $\gamma_1 \geq 1$
such that the following assertion holds:
For any initial data $(\bu_0, \bQ_0) \in D_{q, p} (\BR^N)$
and functions in the right-hand sides 
$(\bff, \bg) \in L_{p, \gamma_1}(\BR_+, W_q^{0, 1}(\BR^N))$,
problem \eqref{l0} admits 
unique solutions $\bu$ and $\bQ$ with
\begin{align*}
&\bu \in W^1_{p, \gamma_1} (\BR_+, L_q(\BR^N)^N) 
\cap L_{p, \gamma_1} (\BR_+, W^2_q(\BR^N)^N),\\
&\bQ \in W^1_{p, \gamma_1} (\BR_+, W^1_q(\BR^N)^{N^2}) 
\cap L_{p, \gamma_1} (\BR_+, W^3_q(\BR^N)^{N^2}), 
\end{align*}
possessing the estimate 
\begin{equation}\label{mr es}
\begin{aligned}
&\|e^{-\gamma t}\pd_t \bu\|_{L_p(\BR_+, L_q(\BR^N))}
+
\|e^{-\gamma t} \bu\|_{L_p(\BR_+, W^2_q(\BR^N))}\\
&+
\|e^{-\gamma t}\pd_t \bQ\|_{L_p(\BR_+, W^1_q(\BR^N))}
+
\|e^{-\gamma t} \bQ\|_{L_p(\BR_+, W^3_q(\BR^N))}\\
&\leq
C_{p, q, N, \gamma_1}
\left(\|(\bu_0, \bQ_0)\|_{D_{q, p} (\BR^N)}
+\|(e^{-\gamma t}\bff, e^{-\gamma t}\bg)\|_{L_p(\BR_+, W^{1, 0}_q(\BR^N))}\right)
\end{aligned}
\end{equation}
for any $\gamma \geq \gamma_1$.
\end{thm}

\subsection{$\CR$-boundedness of solution operators}
In this  subsection, we analyze the following resolvent problem 
in order to prove Theorem \ref{thm:mr}.

\begin{equation}\label{r}\left\{
\begin{aligned}
&\lambda \bu -\Delta \bu + \nabla \fp +\beta \DV \Delta \bQ = \bff,
\enskip \dv \bu = 0
& \quad&\text{in $\R^N$}, \\
&\lambda \bQ - \beta \bD(\bu) -\Delta \bQ + a\left(\bQ - \frac{1}{N} \tr \bQ \bI \right)
=\bg & \quad&\text{in $\R^N$}.
\end{aligned}\right.
\end{equation}
Here, $\lambda$ is the resolvent parameter varying in a sector
\[
\Sigma_{\epsilon, \lambda_0} 
=\{\lambda \in \C \mid |\arg \lambda| < \pi - \epsilon, 
|\lambda| \geq \lambda_0\} 
\]
for $0 < \epsilon < \pi/2$ and $\lambda_0 \geq 1$.
Moreover, we set an angle
$\sigma_0 \in (0, \pi/2)$ by
\begin{equation}\label{sigma}
\sigma_0 =
\left\{
\begin{aligned}
&0 & \text{if } \beta = 0,\\
&\arg(1 + i|\beta|)& \text{if } \beta \neq 0.
\end{aligned}
\right.
\end{equation}

We introduce
the definition of $\CR$-boundedness of operator families.

\begin{dfn}\label{dfn2}
A family of operators $\CT \subset \CL(X,Y)$ is called $\CR$-bounded 
on $\CL(X,Y)$, if there exist constants $C > 0$ and $p \in [1,\infty)$ 
such that for any $n \in \BN$, $\{T_{j}\}_{j=1}^{n} \subset \CT$,
$\{f_{j}\}_{j=1}^{n} \subset X$ and sequences $\{r_{j}\}_{j=1}^{n}$
 of independent, symmetric, $\{-1,1\}$-valued random variables on $[0,1]$, 
we have  the inequality:
$$
\bigg \{ \int_{0}^{1} \|\sum_{j=1}^{n} r_{j}(u)T_{j}f_{j}\|_{Y}^{p}\,du
 \bigg \}^{1/p} \leq C\bigg\{\int^1_0
\|\sum_{j=1}^n r_j(u)f_j\|_X^p\,du\biggr\}^{1/p}.
$$ 
The smallest such $C$ is called $\CR$-bound of $\CT$, 
which is denoted by $\CR_{\CL(X,Y)}(\CT)$.
\end{dfn}
The following theorem is the main result of this subsection.

\begin{thm}\label{thm:Rbdd}
Let $1 < q < \infty$.
Then, 
for any $\sigma \in (\sigma_0, \pi/2)$,
there exist a positive constant $\lambda_0=\lambda_0(\sigma) \geq 1$
and operator families 
\begin{align*}
&\CA (\lambda) \in 
{\rm Hol} (\Sigma_{\sigma, \lambda_0}, 
\CL(W^{0, 1}_q(\R^N), W^2_q(\R^N)^N))\\
&\CB (\lambda) \in 
{\rm Hol} (\Sigma_{\sigma, \lambda_0}, 
\CL(W^{0, 1}_q(\R^N), W^3_q(\R^N)^{N^2}))
\end{align*}
such that 
for any $\lambda = \gamma + i\tau \in \Sigma_{\sigma, \lambda_0}$
, $\bff \in L_q(\BR^N)^N$, and $\bg \in W^1_q(\R^N)^{N^2}$, 
\begin{equation*}
\bu = \CA (\lambda) (\bff, \bg), \quad
\bQ = \CB (\lambda) (\bff, \bg)
\end{equation*}
are unique solutions of problem \eqref{r},
and 
\begin{align}
&\CR_{\CL(W^{0, 1}_q(\R^N), A_q(\R^N))}
(\{(\tau \pd_\tau)^n \CS_\lambda \CA (\lambda) \mid 
\lambda \in \Sigma_{\sigma, 0}\}) 
\leq r_{N, q}, \label{rbdd u}\\
&\CR_{\CL(W^{0, 1}_q(\R^N), B_q(\R^N))}
(\{(\tau \pd_\tau)^n \CT_\lambda \CB (\lambda) \mid 
\lambda \in \Sigma_{\sigma, \lambda_0}\}) 
\leq r_{N, q} \label{rbdd q}
\end{align}
for $n = 0, 1,$
where 
$\CS_\lambda \bu = (\nabla^2 \bu, \lambda^{1/2}\nabla \bu, \lambda \bu)$,
$\CT_\lambda \bQ = (\nabla^3 \bQ, \lambda^{1/2}\nabla^2 \bQ, \lambda \bQ)$,
$A_q(\BR^N) = L_q(\BR^N)^{N^3 + N^2+N}$,
$B_q(\BR^N) = L_q(\BR^N)^{N^5 + N^4} \times W^1_q(\BR^N)^{N^2}$,
and $r_{N, q}$ is a constant independent of $\lambda$.
\end{thm}

Postponing the proof of Theorem \ref{thm:Rbdd}, 
we are concerned with time dependent problem 
\eqref{l0}.
Set
\[
X_q(\R^N)= J_q(\R^N) \times W^1_q(\R^N)^{N^2}.
\]
Let $\CA$ be a linear operator defined by 
\begin{equation}\label{op}
\CA (\bu, \bQ) 
= \left(
 P\Delta \bu - \beta P(\DV \Delta \bQ), 
\beta \bD(\bu) + \Delta \bQ - a \left(\bQ + \frac{1}{N}\tr \bQ \bI\right)
\right)
\end{equation}
for $(\bu, \bQ) \in D(\CA)$,
where $P$ denotes solenoidal projection
and
\[
D(\CA) = (W^2_q(\R^N)^N \cap J_q(\R^N)) \times W^3_q(\R^N)^{N^2}
.\]
Since Definition \ref{dfn2} with $n = 1$ 
implies the uniform boundedness
of the operator family $\CT$, 
 solutions
$\bu$ and $\bQ$ of equations \eqref{r} satisfy 
the resolvent estimate:
\begin{equation}\label{resolvent es}
|\lambda|\|(\bu, \bQ)\|_{W_q^{0, 1}(\BR^N)} 
+|\lambda|^{1/2} \|(\nabla \bu, \nabla^2 \bQ)\|_{L_q(\BR^N)}
+ \|(\bu, \bQ)\|_{W^{2, 3}_q(\BR^N)}
\leq
C_{r_{N, q}}
\|(\bff, \bg)\|_{W^{0, 1}_q(\BR^N)}
\end{equation}
for any $\lambda \in \Sigma_{\sigma, \lambda_0}$ 
and $(\bff, \bg) \in X_q(\BR^N)$.
By \eqref{resolvent es}, we have the following theorem.

\begin{thm}\label{thm:semi1}
Let $1 < q < \infty$.
Then, the operator $\CA$ generates an analytic
semigroup $\{e^{\CA t}\}_{t\geq 0}$ on $X_q(\BR^N)$.  
Moreover, there exist
constants $\gamma_1 \geq 1$ and $C_{q, N, \gamma_1} > 0$
such that $\{e^{\CA t}\}_{t\geq 0}$ satisfies the estimates: 
\begin{align*}
\|e^{\CA t} (\bu_0, \bQ_0) \|_{X_q(\BR^N)}
&\leq C_{q, N, \gamma_1} e^{\gamma_1 t} \|(\bu_0, \bQ_0)\|_{X_q(\BR^N)},\\
\|\pd_t e^{\CA t} (\bu_0, \bQ_0) \|_{X_q(\BR^N)}
&\leq C_{q, N, \gamma_1} e^{\gamma_1 t} t^{-1} \|(\bu_0, \bQ_0)\|_{X_q(\BR^N)},\\
\|\pd_t e^{\CA t} (\bu_0, \bQ_0) \|_{X_q(\BR^N)}
&\leq C_{q, N, \gamma_1} e^{\gamma_1 t} \|(\bu_0, \bQ_0)\|
_{D(\CA) (\BR^N)}
\end{align*}
for any $t > 0$.
\end{thm}
Combining Theorem \ref{thm:semi1} with a real interpolation method
(cf. Shibata and Shimizu \cite[Proof of Theorem 3.9]{SS2}), we have
the following result for equations \eqref{l0} 
with $(\bff, \bg) = (0, O)$.

\begin{thm}\label{thm:semi2}
Let $1 < p, q < \infty$.  
Then, for any $(\bu_0, \bQ_0) \in 
D_{q, p} (\BR^N)$, 
problem \eqref{l0} with $(\bff, \bg) = (0, O)$
admits a unique solution $(\bu, \bQ) = e^{\CA t} (\bu_0, \bQ_0)$
possessing the estimate:
\begin{equation}\label{semi2}
\begin{aligned}
&\|e^{-\gamma t} \pd_t \bu\|_{L_p(\R_+, L_q(\BR^N))}
+ \|e^{-\gamma t}\bu\|_{L_p(\R_+, W^2_q(\BR^N))}\\
&+ \|e^{-\gamma t}\pd_t \bQ\|_{L_p(\R_+, W^1_q(\BR^N))}
+ \|e^{-\gamma t}\bQ\|_{L_p(\R_+, W^3_q(\BR^N))}\\
&
\leq C_{p, q, N, \gamma_1} 
\|(\bu_0, \bQ_0)\|
_{D_{q, p} (\BR^N)}
\end{aligned}
\end{equation}
for any $\gamma \geq \gamma_1$.

\end{thm}

The remaining part of this subsection
is devoted to proving Theorem \ref{thm:Rbdd}. 
For this purpose, we first calculate a solution formula of \eqref{r}.
Taking divergence of the first equation of \eqref{r},
we have 
\begin{equation}\label{p}
\fp=-\beta \dv \DV \bQ + \Delta^{-1}\dv \bff.
\end{equation}
Inserting \eqref{p} into the first equation of \eqref{r},
we have
\begin{equation}\label{uq1}
(\lambda - \Delta)\bu + \beta (\DV \Delta \bQ - \nabla \dv \DV \bQ)
= \bff -\nabla \Delta^{-1}\dv \bff.
\end{equation}
Let $\bold{g} = \bff -\nabla \Delta^{-1}\dv \bff.$
Substituting the formula:
$\tr \bQ = (\lambda - \Delta)^{-1} \tr \bg$,
which obtained by the trace of the second equation of \eqref{r},
into the second equation of \eqref{r} yields
\begin{equation*}\label{uq2}
(\lambda - (\Delta - a))\bQ - \beta \bD(\bu)
=\bg + \frac{a}{N}(\lambda - \Delta)^{-1}\tr \bg \bI.
\end{equation*} 
Thus, setting $\bH = \bg + a(\lambda - \Delta)^{-1}\tr \bg \bI/N$,
we have 
\begin{equation}\label{uq2}
(\lambda - (\Delta - a)) \bQ = \beta \bD(\bu) + \bH.
\end{equation}
Thanks to \eqref{uq2}, $\bQ$ can be represented by $\bu$;
therefore, we first consider a solution formula for $\bu$ below.
Applying $(\lambda - (\Delta - a))$ to \eqref{uq1},
we have
\begin{equation}\label{uq3}
(\lambda - \Delta)(\lambda - (\Delta - a))\bu
+ \beta(\lambda - (\Delta - a)) (\DV \Delta \bQ - \nabla \dv \DV \bQ)
= (\lambda - (\Delta - a))\bold{g}.
\end{equation}
By \eqref{uq2} and $\dv \bu =0$, the second term of \eqref{uq3} can be calculated as follows:
\[
\beta(\lambda - (\Delta - a)) (\DV \Delta \bQ - \nabla \dv \DV \bQ)
=\beta^2 \Delta^2 \bu + \beta(\DV \Delta \bH - \nabla \dv \DV \bH),
\]
so that \eqref{uq3} is equivalent to
\[
P_2(\lambda) \bu = (\lambda - (\Delta - a))\bold{g} - \beta(\DV \Delta \bH - \nabla \dv \DV \bH),
\]
where 
$P_2(\lambda) = (\lambda - \Delta) (\lambda - (\Delta - a)) + \beta^2 \Delta^2$.
Noting that $\DV \Delta \bH - \nabla \dv \DV \bH = \DV \Delta \bg - \nabla \dv \DV \bg$,
we have
\begin{equation}\label{u}
\bu = P_2(\lambda)^{-1}\{(\lambda - (\Delta - a))(\bff - \nabla \Delta^{-1} \dv\bff) 
- \beta(\DV \Delta \bg - \nabla \dv \DV \bg)\}.
\end{equation}
Thus $\bu=(u_1,\dots, u_N)$ has form:
\begin{equation}\label{u1}
\begin{aligned}
u_j &= 
\CA_j(\lambda)(\bff, \bg)
\end{aligned}\end{equation}
with 
\begin{equation}\label{sol.op.u}\begin{aligned}
\CA_j(\lambda)(\bff, \bg)
&= \CF^{-1} \left[ \frac{\lambda + |\xi|^2 + a}{P_2(\xi, \lambda)} 
\left(\hat{f}_j - \frac{\xi_j}{|\xi|^2} \xi \cdot \hat{\bff} \right)\right] \nonumber \\
& \enskip 
- \CF^{-1} \left[ \frac{\beta}{P_2(\xi, \lambda)}
\left(\sum^N_{\ell=1} i\xi_\ell |\xi|^2 \hat{G}_{j\ell} + i \xi_j \sum^N_{\ell, m=1} 
\xi_\ell \xi_m \hat{G}_{\ell m}\right)\right],
\end{aligned}
\end{equation}
where
\begin{equation}\label{p2}
P_2(\xi,  \lambda) = (\lambda + |\xi|^2) (\lambda + |\xi|^2 + a) + \beta^2|\xi|^4. 
\end{equation}
Moreover, using \eqref{uq2}, $\bQ$ with $(j, k)$ component
$Q_{jk}$ is represented as follows:
\begin{equation}\label{q1}
\begin{aligned}
Q_{jk} &= \CF^{-1} \left[\frac{\beta}{\lambda 
+ |\xi|^2+a} (i\xi_j \hat{u}_k + i\xi_k \hat{u}_j)\right]
+ \CF^{-1} \left[\frac{1}{\lambda + |\xi|^2 + a} \hat{G}_{jk} \right]
+\frac{a}{N} \CF^{-1} \left[\frac{1}{P_1(\xi, \lambda)} \tr \hat{\bg} \delta_{jk} \right]\\
&
= \CB_{jk}(\lambda)(\bff, \bg)
\end{aligned}\end{equation} 
with
\begin{equation}\label{sol.op.q}\begin{aligned}
\CB_{jk}(\lambda)(\bff, \bg)
&
= \CF^{-1} \left[\frac{\beta}{P_2(\xi, \lambda)} (i\xi_j f_k + i\xi_k f_j)\right]
- \CF^{-1} \left[\frac{ 2\beta}{P_2(\xi, \lambda)}
\frac{i \xi_j\xi_k}{|\xi|^2 } 
\xi \cdot \hat{\bff} \right] \\
&\enskip
+ \CF^{-1} \left[\frac{\beta^2|\xi|^2}{(\lambda + |\xi|^2 + a)P_2(\xi, \lambda)} 
\left( \xi_j \sum^N_{\ell=1} \xi_\ell \hat{G}_{k\ell}
+ \xi_k\sum^N_{\ell=1} \xi_\ell \hat{G}_{j\ell} \right) \right]  \\
&\enskip
+ \CF^{-1} \left[ \frac{2\beta^2}{(\lambda + |\xi|^2 + a)P_2(\xi, \lambda)}
\xi_j \xi_k \sum^N_{\ell, m=1} \xi_\ell \xi_m \hat{G}_{\ell m} \right] \\
&\enskip
+ \CF^{-1} \left[\frac{1}{\lambda + |\xi|^2 + a} \hat{G}_{jk} \right]
+\frac{a}{N} \CF^{-1} \left[\frac{1}{P_1(\xi, \lambda)} \tr \hat{\bg} \delta_{jk} \right]
\end{aligned}
\end{equation}
where
\begin{equation}\label{p1}
P_1(\xi, \lambda) 
= (\lambda + |\xi|^2) (\lambda + |\xi|^2 + a).
\end{equation}
Let $\CA(\lambda) (\bff, \bg)$ be a vector whose 
$j^{\rm th}$ component is 
 $\CA_j(\lambda)(\bff, \bg)$ and
let $\CB(\lambda)(\bff, \bg)$ be a matrix whose 
$(j, k)^{\rm th}$ component is 
by $\CB_{jk}(\lambda)(\bff, \bg)$, respectively. 
To prove the $\CR$-boundedness of $\CA(\lambda)$ and $\CB(\lambda)$,
we introduce the following two lemmas.
Lemma \ref{lem:5.3} was proved by \cite[Proposition 3.4]{DHP}
and 
Lemma \ref{kernel} proved by 
\cite[Lemma 2.1]{DS}, 
\cite[Theorem 3.3]{ES}, and 
\cite[Lemma 2.5]{Sa}

\begin{lem}\label{lem:5.3}
$\thetag1$ 
Let $X$ and $Y$ be Banach spaces, 
and let $\CT$ and $\CS$ be $\CR$-bounded families in $\CL(X, Y)$. 
Then, $\CT+\CS=\{T+S \mid T\in \CT, S\in \CS\}$ is also 
$\CR$-bounded family in $\CL(X, Y)$ and 
\[
\CR_{\CL(X, Y)}(\CT+\CS)\leq \CR_{\CL(X, Y)}(\CT)
+\CR_{\CL(X, Y)}(\CS).
\]
$\thetag2$
Let $X$, $Y$ and $Z$ be Banach spaces and 
let $\CT$ and $\CS$ be $\CR$-bounded families
 in $\CL(X, Y)$ and $\CL(Y, Z)$, respectively. 
Then, $\CS\CT=\{ST \mid T\in \CT, S\in \CS\}$ is also 
an $\CR$-bounded family 
in $\CL(X, Z)$ and 
\[
\CR_{\CL(X, Z)}(\CS\CT)\leq \CR_{\CL(X, Y)}(\CT)\CR_{\CL(Y, Z)}(\CS). 
\]
\end{lem}

\begin{lem}\label{kernel}
 Let $1 < q < \infty$, $\delta > 0$.
Assume that $k(\xi, \lambda)$, $\ell(\xi, \lambda)$,
and $m(\xi, \lambda)$ 
are functions on $(\BR^N \setminus \{0\}) \times \Sigma_{\sigma, 0}$ such 
that for any $\sigma \in (\sigma_0, \pi/2)$ and
any multi-index $\alpha \in \BN^N_0$
there exists a positive constant $M_{\alpha, \sigma}$ such that 
\begin{align*} 
&|\pd_\xi^\alpha k(\xi, \lambda)|
\leq M_{\alpha, \epsilon} |\xi|^{1-|\alpha|}, \quad
|\pd_\xi^\alpha \ell(\xi, \lambda)|
\leq M_{\alpha, \epsilon} |\xi|^{-|\alpha|},\\
&
|\pd_\xi^\alpha m(\xi, \lambda)|
\leq M_{\alpha, \epsilon} (|\lambda|^{1/2} + |\xi|)^{-1} |\xi|^{-|\alpha|}
\end{align*}
for any $(\xi, \lambda) \in (\BR^N\setminus\{0\}) \times \Sigma_{\sigma, 0}$.
Let $K(\lambda)$, $L(\lambda)$, and $M(\lambda)$
 be operators defined by
\[
\begin{aligned}
&[K(\lambda) f](x)
= \CF^{-1}_\xi[k(\xi, \lambda)\hat{f}(\xi)](x)& \enskip& (\lambda \in \Sigma_{\sigma, 0}),\\
&[L(\lambda) f](x)
= \CF^{-1}_\xi[\ell(\xi, \lambda)\hat{f}(\xi)](x) &\enskip& (\lambda \in \Sigma_{\sigma, 0}),\\
&[M(\lambda) f](x)
= \CF^{-1}_\xi[m(\xi, \lambda)\hat{f}(\xi)](x) &\enskip &(\lambda \in \Sigma_{\sigma, \delta}).
\end{aligned}
\]
Then, the following assertions hold true: 
\begin{itemize}
\item[\thetag1]
The set
$\{K(\lambda) \mid \lambda \in \Sigma_{\sigma, 0}\}$ is 
$\CR$-bounded on 
$\CL(W^1_q(\BR^N), L_q(\BR^N))$
and there exists a positive constant $C_{N, q}$ such that 
\[
\CR_{\CL(W^1_q(\BR^N), L_q(\BR^N))}
(\{K(\lambda) \mid \lambda \in \Sigma_{\sigma, 0} \})
\leq C_{N, q}\max_{|\alpha| \leq N+1} M_{\alpha, \sigma}.
\]
\item[\thetag2]
Let $n=0, 1$. Then, the set
$\{L(\lambda) \mid \lambda \in \Sigma_{\sigma, 0}\}$ is 
$\CR$-bounded on 
$\CL(W^n_q(\BR^N))$
and there exists a positive constant $C_{N, q}$ such that 
\[
\CR_{\CL(W^n_q(\BR^N))}
(\{L(\lambda) \mid \lambda \in \Sigma_{\sigma, 0} \})
\leq C_{N, q}\max_{|\alpha| \leq N+1} M_{\alpha, \sigma}.
\]
\item[\thetag3]
The set
$\{M(\lambda) \mid \lambda \in \Sigma_{\sigma, \delta}\}$ is 
$\CR$-bounded on 
$\CL(L_q(\BR^N), W^1_q(\BR^N))$
and there exists a positive constant $C_{N, q}$ such that 
\[
\CR_{\CL(L_q(\BR^N), W^1_q(\BR^N))}
(\{M_\lambda \mid \lambda \in \Sigma_{\sigma, \delta} \})
\leq C_{N, q, \delta}\max_{|\alpha| \leq N+1} M_{\alpha, \sigma}.
\]
\end{itemize}

\end{lem}
To use Lemma \ref{kernel}, we prepare the following lemmas.

\begin{lem}\label{spectrum}
\begin{itemize} 
\item[\thetag1]
Let $0 < \epsilon < \pi/2$.
Then, for any $\lambda \in \Sigma_{\epsilon, 0}$
and
$\alpha \geq 0$, we have
\begin{equation}\label{spectrum1}
|\lambda + \alpha| \geq \left(\sin \frac{\epsilon}{2}\right) (|\lambda| + \alpha).
\end{equation}
\item[\thetag2]
Let $0 < \epsilon < \pi/2$ and 
$a >0$. Then, for any $(\xi, \lambda) \in \BR^N \times \Sigma_{\epsilon, 0}$,
we have
\[
|P_1(\xi, \lambda)| \geq \left(\frac{1}{2}\sin \frac{\epsilon}{2}\right)^2 (|\lambda|^{1/2}+ |\xi|)^4.
\]
\item[\thetag3]
Let $\sigma_0$ be the angle defined in \eqref{sigma}.
Then, for any $\sigma \in (\sigma_0, \pi /2)$ and $(\xi, \lambda) \in \BR^N \times \Sigma_{\sigma, 0}$,
we have     
\begin{equation}\label{spectrum3}
|P_2(\xi, \lambda)| \geq C_{\sigma, \beta} (|\lambda|^{1/2}+ |\xi|)^4
\end{equation}
with some constant $C_{\sigma, b}$ independent of $\xi$ and $\lambda$.
\end{itemize}
\end{lem}

\begin{proof}
\thetag1 is well-known (cf. e.g. \cite[Lemma 3.5.2]{S}).
\thetag2 follows from \thetag1.
We shall show \thetag3.
In the case $\beta = 0$, \eqref{spectrum3} is proved by \thetag2 because $P_2(\xi, \lambda) = P_1(\xi, \lambda)$.
In the case $\beta \neq 0$, we rewrite $P_2(\xi, \lambda)$ as follows: 
\[
P_2(\xi, \lambda) = (\lambda - \lambda_+)(\lambda - \lambda_-),
\]
where
\begin{align}\label{p2 low}
\left\{
\begin{aligned}
&\lambda_+ = -|\xi|^2 + O(|\xi|^4), \\
&\lambda_- = -|\xi|^2 - a + O(|\xi|^4)  \enskip \text{as } |\xi| \to 0,
\end{aligned}
\right.
\end{align}
\begin{equation}\label{p2 high}
\lambda_\pm 
= - (1 \pm i|\beta|)|\xi|^2 +O(1) \enskip \text{as } |\xi| \to \infty.
\end{equation}
By \eqref{p2 low},
we can prove \eqref{spectrum3} for any $\sigma \in (0, \pi/2)$
by the same way as $\beta=0$.
We consider the high frequency part.
Let $\lambda = |\lambda| e^{i\theta}$
for $|\theta| \leq \pi -\sigma$, $\sigma \in (\sigma_0, \pi/2)$.
Noting that 
$\lambda_\pm = -\sqrt{1+\beta^2}|\xi|^2 e^{\pm i \sigma_0}+O(1)$
for large $|\xi|$,
we have
\begin{align*}
|\lambda - \lambda_\pm|^2
&\geq |\lambda|^2 + (1+\beta^2)|\xi|^4 
+ 2 \sqrt{1+\beta^2}|\lambda||\xi|^2 \cos (\theta \mp \sigma_0)
-(|\lambda|+|\xi|^2+O(1))O(1)\\
&\geq  |\lambda|^2 + (1+\beta^2)|\xi|^4 - 2 \sqrt{1+\beta^2}|\lambda||\xi|^2 \cos (\sigma - \sigma_0) -(|\lambda|+|\xi|^2+O(1))O(1)\\
&\geq (1- \cos (\sigma - \sigma_0))(|\lambda|^2 + |\xi|^4)
-(|\lambda|+|\xi|^2+O(1))O(1).
\end{align*}
Here, we used that
$\cos (\theta \mp \sigma_0) \geq \cos \{\pi - (\sigma - \sigma_0)\} 
=- \cos (\sigma - \sigma_0)$.
Therefore, we have \eqref{spectrum3} for any $\sigma \in (\sigma_0, \pi/2)$
and $\lambda \in \Sigma_{\sigma, 0}$. 
\end{proof}
Using Lemma \ref{spectrum} and 
the following Bell's formula for the derivatives of the composite functions:
\[
\pd^\alpha_\xi f (g (\xi))
= \sum^{|\alpha|}_{k = 1} f^{(k)} (g(\xi)) 
\sum_{\alpha = \alpha_1 + \cdots + \alpha_k \atop |\alpha_i| \geq 1}
\Gamma^{\alpha}_{\alpha_1, \ldots, \alpha_k}
(\pd_{\xi}^{\alpha_1} g(\xi)) \cdots
(\pd_{\xi}^{\alpha_k} g(\xi))
\]
with $f^{(k)}(t) = d^k f(t)/dt^{k}$ and suitable coefficients
 $\Gamma^{\alpha}_{\alpha_1, \ldots, \alpha_k}$,
we have the following estimates.

\begin{lem}\label{lem:spectrum2}
\begin{itemize} 
\item[\thetag1]
Let $0 < \epsilon < \pi/2$, $a>0$ and $n = 0, 1$.
Then, for any multi-index $\alpha \in \BN^N_0$,
there exists a positive constant $C$
depending on at most $\alpha$, $\epsilon$
and $b$ such that for any
$(\xi, \lambda) \in \BR^N \times \Sigma_{\epsilon, 0}$
with $\lambda = \gamma + i\tau$,
\begin{align*}
|\pd_\xi^\alpha \{(\tau \pd_\tau)^n (\lambda + |\xi|^2)^{-1}\}|
&\leq C (|\lambda|^{1/2} + |\xi|)^{-2-|\alpha|},\\
|\pd_\xi^\alpha \{(\tau \pd_\tau)^n (\lambda + |\xi|^2+a)^{-1}\}|
&\leq C (|\lambda|^{1/2} + |\xi|)^{-2-|\alpha|},\\
|\pd_\xi^\alpha \{(\tau \pd_\tau)^n P_1(\xi, \lambda)^{-1}\}|
&\leq C (|\lambda|^{1/2} + |\xi|)^{-4-|\alpha|}.
\end{align*} 
\item[\thetag2]
Let $a>0$ and $n = 0, 1$.
Then, for any $\sigma \in (\sigma_0, \pi/2)$ and 
any multi-index $\alpha \in \BN^N_0$,
there exists a positive constant $C$ depending on
at most $\alpha$, $\epsilon$
and $b$ such that for any
$(\xi, \lambda) \in \BR^N \times \Sigma_{\sigma, 0}$
with $\lambda = \gamma + i\tau$,
\[
|\pd_\xi^\alpha \{(\tau \pd_\tau)^n P_2(\xi, \lambda)^{-1}\}|
\leq C (|\lambda|^{1/2} + |\xi|)^{-4-|\alpha|}.
\]
\end{itemize}
\end{lem}

\begin{proof}[Proof of Theorem \ref{thm:Rbdd}]
Let $n = 0, 1$, $j, k, \ell, m =1, \dots, N$ and let $\alpha$ be any multi-index of $\BN_0^N$.
We first prove $\CR$-boundedness of $\CS_\lambda \CA(\lambda)$.
For this purpose, 
we verify $(\lambda, \lambda^{1/2}\xi_a, \xi_a \xi_b)\CA(\lambda)\bff$
satisfies the assumption of Lemma \ref{kernel}.
Using Lemma \ref{lem:spectrum2} and Leibniz's rule, for any 
$(\xi, \lambda) \in (\BR^N \setminus \{0\}) \times \Sigma_{\sigma, 0}$,
we have
\begin{align*}
\left|
\pd_\xi^\alpha (\tau \pd_\tau)^n
\left\{(\xi_a \xi_b, \lambda^{1/2}\xi_a, \lambda)\frac{\lambda + |\xi|^2 + a}{P_2(\xi, \lambda)}\right\}
\right|
&\leq C_{\alpha, \epsilon}|\xi|^{-|\alpha|},\\
\left|
\pd_\xi^\alpha (\tau \pd_\tau)^n
\left\{(\xi_a \xi_b, \lambda^{1/2}\xi_a, \lambda)
\frac{(\lambda + |\xi|^2 + a)\xi_j \xi_k}{P_2(\xi, \lambda)|\xi|^2}
\right\}
\right|
&\leq C_{\alpha, \epsilon}|\xi|^{-|\alpha|},\\
\left|
\pd_\xi^\alpha (\tau \pd_\tau)^n
\left\{\frac{(\xi_a \xi_b, \lambda^{1/2}\xi_a, \lambda)\xi_j|\xi|^2}{P_2(\xi, \lambda)}\right\}
\right|
&\leq C_{\alpha, \epsilon}|\xi|^{1-|\alpha|},\\
\left|
\pd_\xi^\alpha (\tau \pd_\tau)^n
\left\{\frac{(\xi_a \xi_b, \lambda^{1/2}\xi_a, \lambda)\xi_j \xi_k \xi_\ell}{P_2(\xi, \lambda)}\right\}
\right|
&\leq C_{\alpha, \epsilon}|\xi|^{1-|\alpha|}
\end{align*}
which combined with Lemma \ref{kernel} \thetag2 
yields \eqref{rbdd u} in Theorem \ref{thm:Rbdd}.

We next consider $\CR$-boundedness of $\CT_\lambda \CB(\lambda)$.
Using Lemma \ref{spectrum2} and Leibniz's rule, for any 
$(\xi, \lambda) \in (\BR^N \setminus \{0\}) \times \Sigma_{\sigma, \lambda_0}$,
we have
\begin{align*}
\left|
\pd_\xi^\alpha (\tau \pd_\tau)^n
\left\{
\frac{(\xi_a \xi_b \xi_c, \lambda^{1/2}\xi_a \xi_b)\xi_j}{P_2(\xi, \lambda)}
\right\}
\right|
&\leq C_{\alpha, \epsilon}|\xi|^{-|\alpha|},\\
\left|
\pd_\xi^\alpha (\tau \pd_\tau)^n
\left\{
\frac{(\xi_a \xi_b \xi_c, \lambda^{1/2}\xi_a \xi_b)\xi_j \xi_k\xi_\ell}{P_2(\xi, \lambda)|\xi|^2}
\right\}
\right|
&\leq C_{\alpha, \epsilon}|\xi|^{-|\alpha|},\\
\left|
\pd_\xi^\alpha (\tau \pd_\tau)^n
\left\{
\frac{(\xi_a \xi_b \xi_c, \lambda^{1/2}\xi_a \xi_b)|\xi|^2\xi_j\xi_k}{(\lambda + |\xi|^2 +a)P_2(\xi, \lambda)}
\right\}
\right|
&\leq C_{\alpha, \epsilon}|\xi|^{1-|\alpha|},\\
\left|
\pd_\xi^\alpha (\tau \pd_\tau)^n
\left\{
\frac{(\xi_a \xi_b \xi_c, \lambda^{1/2}\xi_a \xi_b)\xi_j \xi_k \xi_\ell \xi_m}
{(\lambda + |\xi|^2 +a)P_2(\xi, \lambda)}
\right\}
\right|
&\leq C_{\alpha, \epsilon}|\xi|^{1-|\alpha|},\\
\left|
\pd_\xi^\alpha (\tau \pd_\tau)^n
\left\{
\frac{(\xi_a \xi_b \xi_c, \lambda^{1/2}\xi_a \xi_b)}{\lambda + |\xi|^2+a}
\right\}
\right|
&\leq C_{\alpha, \epsilon}|\xi|^{1-|\alpha|},\\
\left|
\pd_\xi^\alpha (\tau \pd_\tau)^n
\left\{
\frac{(\xi_a \xi_b \xi_c, \lambda^{1/2}\xi_a \xi_b)}{P_1(\xi, \lambda)}
\right\}
\right|
&\leq C_{\alpha, \epsilon}|\xi|^{-|\alpha|},
\end{align*}
which combined with Lemma \ref{kernel} \thetag1, \thetag2,
and Lemma \ref{lem:5.3} 
gives
\begin{equation}\label{rbdd b1}
\begin{aligned}
\CR_{\CL(W^{0, 1}_q(\BR^N), L_q(\BR^N)^{N^5})}
(\{(\tau \pd_\tau)^n (\nabla^3 \CB(\lambda)) \mid \lambda \in \Sigma_{\sigma, \lambda_0}\})
&\leq C_{N, q}, \\
\CR_{\CL(W^{0, 1}_q(\BR^N), L_q(\BR^N)^{N^4})}
(\{(\tau \pd_\tau)^n (\lambda ^{1/2} \nabla^2 \CB (\lambda)) \mid \lambda \in \Sigma_{\sigma, \lambda_0}\})
&\leq C_{N, q}.
\end{aligned}
\end{equation}
Analogously, 
for any 
$(\xi, \lambda) \in (\BR^N \setminus \{0\}) \times \Sigma_{\sigma, \lambda_0}$,
we have
\begin{align*}
\left|
\pd_\xi^\alpha (\tau \pd_\tau)^n
\left\{
\frac{\lambda \xi_j}{P_2(\xi, \lambda)}
\right\}
\right|
&\leq C_{\alpha, \epsilon}(|\lambda|^{1/2} + |\xi|)^{-1}
|\xi|^{-|\alpha|},\\
\left|
\pd_\xi^\alpha (\tau \pd_\tau)^n
\left\{
\frac{\lambda \xi_j \xi_k\xi_\ell}{P_2(\xi, \lambda)|\xi|^2}
\right\}
\right|
&\leq C_{\alpha, \epsilon}(|\lambda|^{1/2} + |\xi|)^{-1}
|\xi|^{-|\alpha|},\\
\left|
\pd_\xi^\alpha (\tau \pd_\tau)^n
\left\{
\frac{\lambda |\xi|^2\xi_j\xi_k}{(\lambda + |\xi|^2 +a)P_2(\xi, \lambda)}
\right\}
\right|
&\leq C_{\alpha, \epsilon}|\xi|^{-|\alpha|},\\
\left|
\pd_\xi^\alpha (\tau \pd_\tau)^n
\left\{
\frac{\lambda \xi_j \xi_k \xi_\ell \xi_m}
{(\lambda + |\xi|^2 +a)P_2(\xi, \lambda)}
\right\}
\right|
&\leq C_{\alpha, \epsilon}|\xi|^{-|\alpha|},\\
\left|
\pd_\xi^\alpha (\tau \pd_\tau)^n
\left\{
\frac{\lambda}{\lambda + |\xi|^2+a}
\right\}
\right|
&\leq C_{\alpha, \epsilon}|\xi|^{-|\alpha|},\\
\left|
\pd_\xi^\alpha (\tau \pd_\tau)^n
\left\{
\frac{\lambda}{P_1(\xi, \lambda)}
\right\}
\right|
&\leq C_{\alpha, \epsilon}(|\lambda|^{1/2} + |\xi|)^{-2}
|\xi|^{-|\alpha|},
\end{align*}
which combined with Lemma \ref{kernel} \thetag2, \thetag3,
and Lemma \ref{lem:5.3}
gives
\begin{equation}\label{rbdd b2}
\CR_{\CL(W^{0,1}_q(\BR^N), W^1_q(\BR^N)^{N^2})}
(\{(\tau \pd_\tau)^n (\lambda \CB(\lambda)) \mid \lambda \in \Sigma_{\sigma, \lambda_0}\})
\leq C_{N, q}.
\end{equation}
\eqref{rbdd b1} and \eqref{rbdd b2} imply \eqref{rbdd q} in Theorem \ref{thm:Rbdd}.

We finally show the uniqueness of solutions.
Let $(\bu, \bQ) \in W^2_q(\BR^N)\times W^3_q(\BR^N)^{N^2}$ 
satisfy the homogeneous equations:
\begin{equation}\label{r1}\left\{
\begin{aligned}
&\lambda \bu -\Delta \bu + \nabla \fp +\beta \dv \Delta \bQ = 0,
\quad \dv \bu=0
& \quad&\text{in $\R^N$}, \\
&\lambda \bQ + \beta \bD(\bu) -\Delta \bQ + a\left(\bQ - \frac{1}{N} \tr \bQ \bI \right)
=O & \quad&\text{in $\R^N$}. \\
&\dv \bu = 0& \quad&\text{in $\R^N$}.\end{aligned}\right.
\end{equation}
Applying formulas \eqref{p}, \eqref{uq1}, \eqref{uq2} and  \eqref{uq3} with
$\bff=0$ and $\BG=0$ yields that $P_2(\lambda)\bu=0$. If we set
$P_2(\xi, \lambda) = (\lambda+|\xi|^2)(\lambda+|\xi|^2+a) +\beta^2|\xi|^4$,
Fourier transformation $\hat \bu$ of $\bu$ satisfies $P_2(\xi, \lambda)\hat \bu
=0$, which, combined with Lemma \ref{spectrum} \thetag3, yields that 
$\bu=0$.  Thus, by \eqref{uq2}, $(\lambda + |\xi|^2+a)\hat \BQ=0$, 
where $\hat \BQ$ denotes Fourier transformation of $\BQ$, 
which yields that $\BQ=0$. By \eqref{p}, we also have $\fp=0$.
 This completes the proof of the uniqueness,
and therefore we have proved Theorem \ref{thm:Rbdd}.

\end{proof}

\subsection{A proof of Theorem \ref{thm:mr}}
To prove Theorem \ref{thm:mr}, the key tool is
the Weis
operator valued Fourier multiplier theorem.
Let $\CD(\BR,X)$ and $\CS(\BR,X)$ be the set of all $X$ 
valued $C^{\infty}$ functions having compact support 
and the Schwartz space of rapidly decreasing $X$ 
valued functions, respectively,
while $\CS'(\BR,X)= \CL(\CS(\BR,\BC),X)$. 
Given $M \in L_{1,\rm{loc}}(\BR \backslash \{0\}, \CL(X, Y))$, 
we define the operator $T_{M} : \CF^{-1} \CD(\BR,X)\rightarrow \CS'(\BR,Y)$ 
by
\begin{align}\label{eqTM}
T_M \phi=\CF^{-1}[M\CF[\phi]],\quad (\CF[\phi] \in \CD(\BR,X)). 
\end{align}

\begin{thm}[Weis \cite{W}]\label{Weis}
Let $X$ and $Y$ be two UMD Banach spaces and $1 < p < \infty$. 
Let $M$ be a function in $C^{1}(\BR \backslash \{0\}, \CL(X,Y))$ such that 
\begin{align*}
\CR_{\CL(X,Y)} (\{(\zeta \frac{d}{d\zeta})^{\ell} M(\zeta) \mid
 \zeta \in \BR \backslash \{0\}\}) \leq \kappa < \infty
\quad (\ell =0,1)
\end{align*}
with some constant $\kappa$. 
Then, the operator $T_{M}$ defined in \eqref{eqTM} 
is extended to a bounded linear operator from
 $L_{p}(\BR,X)$ into $L_{p}(\BR,Y)$. 
Moreover, denoting this extension by $T_{M}$, we have 
\begin{align*}
\|T_{M}\|_{\CL(L_p(\BR,X),L_p(\BR,Y))} \leq C\kappa
\end{align*}
for some positive constant $C$ depending on $p$, $X$ and $Y$. 
\end{thm}

We now prove Theorem \ref{thm:mr}. 
Using the linear operator $\CA$ defined by \eqref{op},
we can rewrite problem \eqref{l0} as follows:
\begin{equation}\label{rl0}
\pd_t \bU - \CA \bU = \bF \text{ in $\R^N$ for $t>0$}, \enskip
\bU|_{t=0} = \bU_0,
\end{equation}
where $\bU=(\bu, \bQ)$, $\bF=(P\bff, \bg)$,
and $\bU_0=(P\bu_0, \bQ_0)$.
Let $\bF \in L_{p, \gamma_1} (\BR_+, W^{0, 1}_q(\BR^N))$
and let $\bF_0$ be the zero extension of $\bF$
to $t < 0$.
Here, $\gamma_1$ is the same as in Theorem \ref{thm:semi1} and
Theorem \ref{thm:semi2}. 
To solve problem \eqref{rl0},  we  consider problem:
\begin{equation}\label{l1}
\pd_t \bU_1 - \CA \bU_1 = \bF_0
\quad \text{ in $\R^N$ for $t \in \BR$}.
\end{equation}
Here,  $\bF_0= (\bff_0, \BG_0)$ and 
$\bff_0$ and $\BG_0$ are zero extensions of $\bff$ and $\BG$ to 
$t < 0$. To solve equations \eqref{l1}, 
we introduce Laplace transformation $\CL$ and Laplace inverse 
transformation $\CL^{-1}$ defined by 
$$\CL[f](\lambda) = \int_{\BR}e^{-\lambda t}f(t)\,dt, \quad
 \CL^{-1}[g](t)= \frac{1}{2\pi}\int_{\BR}e^{\lambda t}g(\tau)\,d\tau$$
where $\lambda = \gamma + i\tau \in\BC$, which are written by Fourie transformation
$\CF$ and Fourier inverse transformation in $\BR$ as 
$$\CL[f](\lambda) = \CF[e^{-\gamma t}f(t)](\tau),
\quad \CL^{-1}[g](t) = e^{\gamma t}\CF^{-1}[g](\tau).$$
Applying Laplace transformation to equations \eqref{l1} yields
$$
\lambda \CL[\bU_1] -\CA \CL[\bU_1] = \CL[\bF_0] 
\quad\text{in $\R^N$}. 
$$
Applying the operators $\CA(\lambda)$ and $\CB(\lambda)$ given in
Theorem \ref{thm:Rbdd} yields that
$$\CL[\bU_1](\lambda) = (\CA(\lambda)\CL[\bF_0](\lambda), 
\CB(\lambda)\CL[\bF_0](\lambda)).
$$
Since $\bF_0 = 0$ for $t < 0$, $\CL[\bF_0](\lambda)$ is holomorphic for 
${\rm Re}\,\lambda > 0$, and so by Cauchy's theorem in theory of one complex
variable we see that
\begin{equation}\label{ident:1}
\CL[\bU_1](\gamma_1 + i\tau) = \CL[\bU_1](\gamma + i\tau)
\quad (\gamma > \lambda_0, \enskip \tau \in \BR).
\end{equation}
Using Laplace inverse transformation, we define $\bU_1$ by 
$$\bU_1(\cdot, t) = \CL^{-1}[(\CA(\lambda)\CL[\bF_0](\lambda), 
\CB(\lambda)\CL[\bF_0](\lambda))](t).$$
Since we can write
$$e^{-\gamma t}\bU_1 = \CF^{-1}[(\CA(\lambda)\CF[e^{-\gamma t}\bF_0](\tau), 
\CB(\lambda)\CF[e^{-\gamma t}\bF_0](\tau))]
$$
using Theorem \ref{thm:Rbdd} and applying Theorem \ref{Weis}, 
we have
\begin{equation}\label{ident:2}
\|e^{-\gamma t}\bU_1\|_{L_p(\BR, W^{2,3}_q(\BR^N))} 
+ \|e^{-\gamma t}\pd_t\bU_1\|_{L_p(\BR, W^{0,1}_q(\BR^N))}
\leq C\|e^{-\gamma t}\bF\|_{L_p(\BR_+, W^{0,1}_q(\BR^N))}.
\end{equation}
Since $|(\tau\pd_\tau)^\ell\gamma/\lambda| \leq C$ for $\lambda
\in \Sigma_{\epsilon, \lambda_0}$, we have
$$\gamma \|e^{-\gamma t}\bU_1\|_{L_p(\BR, W^{0,1}_q(\BR^N))}
\leq C\|e^{-\gamma t}\pd_t \bU_1\|_{L_p(\BR, W^{0,1}_q(\BR^N))}
$$
as follows from \cite[Proposition 3.6 and Corollary 3.7]{DHP},
which, combined with \eqref{ident:2},  yields that
\begin{align*}
&\gamma \|\bU_1\|_{L_p((-\infty, 0), W^{0,1}_q(\BR^N))}
\leq \gamma \|e^{-\gamma t}\bU_1\|_{L_p(\BR, W^{0,1}_q(\BR^N))}
\leq \|e^{-\gamma t}\pd_t\bU_1\|_{L_p(\BR, W^{0,1}_q(\BR^N))} \\
&\quad 
\leq C\|e^{-\gamma t}\pd_t \bU_1\|_{L_p(\BR, W^{0,1}_q(\BR^N))}
\leq C\|e^{-\gamma_1 t}\pd_t\bU_1\|_{L_p(\BR, W^{0,1}_q(\BR^N))}.
\end{align*}
for any $\gamma \geq \gamma_1 >\lambda_0$. Thus, letting $\gamma \to \infty$
yields that $\|\bU_1\|_{L_p((-\infty, 0), W^{0,1}_q(\BR^N))}=0$, 
which implies that $\bU_1(\cdot, t)=0$ for $t < 0$. By trace theorem in
theory of real interpolation, $\bU_1$ is a $D_{q,p}(\BR^N)$ 
valued continuous in $\BR$, we have $\bU_1(\cdot, t)=0$ for $t \leq 0$.

In view of Theorem \ref{thm:semi2}, we set $\bU= \bU_1+e^{\CA t}(\bu_0, \BQ_0)$,
and then $\bU$ is a required solution of equations \eqref{l0}, which proves
the existence part of Theorem \ref{thm:mr}. 

We finally show the uniqueness of solutions.  
By Theorem \ref{thm:semi1} and Duhamel's principle
we can write $\bU$ as follows:
\[
\bU = e^{\CA t} \bU_0 + \int^t_0 e^{\CA (t-s)} \bF(s)\, ds.
\]
Thus, if $\bU$ satisfies the equation \eqref{rl0} with $\bF = 0$, $\bU_0 = 0$,
we have $\bU=0$.
This completes the proof of Theorem \ref{thm:mr}.


\section{Decay property of solutions to the linearized problem}\label{sec:decay}
In this section, we consider the following linearized problem:
\begin{equation}\label{l3}\left\{
\begin{aligned}
&\pd_t \bu -\Delta \bu + \nabla \fp +\beta \DV \Delta \bQ = 0,
\enskip \dv \bu=0
& \quad&\text{in $\R^N$ for $t \in (0, T)$}, \\
&\pd_t \bQ - \beta \bD(\bu) -\Delta \bQ + a\left(\bQ - \frac{1}{N} \tr \bQ \bI \right)
=O & \quad&\text{in $\R^N$ for $t \in (0, T)$}, \\
&(\bu, \bQ)|_{t=0} = (\bff, \bg)& \quad&\text{in $\R^N$}.
\end{aligned}\right.
\end{equation}
Let $d = \tr \bQ$.
We consider the $L_p$-$L_q$ decay estimates
for $d$, $\bu$, and $\bQ$ satisfying \eqref{l3} below. 

\subsection{Decay estimates for $d$}
We first consider the decay estimates for $d$.
Taking divergence of the second equation of \eqref{l3}
and using $\dv \bu =0$,
$d$ satisfies the heat equations:
\begin{equation}\label{d}
\pd_t d - \Delta d = 0 \enskip \text{in $\R^N$ for $t \in (0, T)$},
\enskip
d|_{t=0} = d_0 \enskip \text{in $\R^N$},
\end{equation}
where $d_0 = \tr \bQ_0$.
By the the $L_p$-$L_q$ estimate for the heat semigroup
we see that $d$ satisfies 
\begin{equation}\label{esd}
\|\nabla^j d \|_{L_p(\R^N)}
\leq C
t^{-\frac{N}{2} (\frac{1}{q} - \frac{1}{p}) - \frac{j}{2}}
\|d_0\|_{L_q(\R^N)}
\end{equation}
for any $t>0$,
where $d_0 \in L_q(\R^N)$, $j \in \N_0$, and $1 \leq q \leq p \leq \infty$.

\subsection{Decay estimates for $\bu$ and $\bQ$}
We next consider the decay estimates for $\bu$ and $\bQ$.
Theorem \ref{thm:semi2} implies that 
there exist operators
\begin{equation}\label{opu}
S(t) \in \CL(X_q(\R^N), W^2_q(\BR^N)^N),
\enskip
T(t) \in \CL(X_q(\R^N), W^3_q(\BR^N)^{N^2}) 
\end{equation}
such that for any $(\bff, \bg) \in X_q(\R^N)$,
$\bu = S(t)(\bff, \bg)$ and $\bQ = T(t)(\bff, \bg)$ satisfy \eqref{l3}.
We shall prove the $L_p$-$L_q$ decay estimates of operators
$S(t)$ and $T(t)$.
For this purpose, we decompose the solution into low and high frequency parts.
Let $\vp \in C^\infty_0 (\BR^N)$ be a function such that
$0 \leq \vp(\xi) \leq 1$, $\vp(\xi) = 1$ if $|\xi| \leq 1/3$ and 
$\vp(\xi) = 0$ if $|\xi| \geq 2/3$.
Let $\vp_0$ and $\vp_\infty$ be functions such that
\[
\vp_0(\xi) = \vp(\xi/A_0), \enskip
\vp_\infty(\xi) = 1 - \vp(\xi/A_0),
\]
where $A_0 \in (0, 1)$ is a  sufficiently small number, 
which determined in subsection \ref{subsection low}.
Moreover, we set 
\begin{align*}
&S_n(t)(\bff, \bg)
=\frac{1}{2\pi i} \sum^2_{j = 1}
\CF^{-1} \left[\int_\Gamma e^{\lambda t} 
\ell_j(\xi, \lambda) \vp_n \hat \bff \,d\lambda \right](x)
+ \frac{1}{2\pi i} \sum^2_{j = 1}\CF^{-1} \left[\int_\Gamma e^{\lambda t} 
m_j(\xi, \lambda) \vp_n \widehat{\DV\bg} \,d\lambda \right](x),\\
&T_n(t)(\bff, \bg)
=\frac{1}{2\pi i} \sum^2_{j = 1}
\CF^{-1} \left[\int_\Gamma e^{\lambda t} 
\tilde \ell_j(\xi, \lambda) \vp_n \hat \bff \,d\lambda \right](x)
+ \frac{1}{2\pi i} \sum^2_{j = 1}\CF^{-1} \left[\int_\Gamma e^{\lambda t} 
\tilde m_j(\xi, \lambda) \vp_n \widehat{\DV\bg} \,d\lambda \right](x)\\
&\enskip + \frac{1}{2\pi i} \CF^{-1} \left[\int_\Gamma e^{\lambda t} 
\tilde m_3(\xi, \lambda) \vp_n \hat{\bg} \,d\lambda \right](x)
+\frac{1}{2\pi i} \frac{a}{N}\CF^{-1} \left[\int_\Gamma e^{\lambda t} 
\tilde m_3(\xi, \lambda) \vp_n{\rm tr}\,\hat \BQ
 \bI \,d\lambda \right](x),
\end{align*}
where $n=0$, $\infty$,
\begin{align*}
&\ell_1(\xi, \lambda) \hat \bff = \frac{\lambda + |\xi|^2 + a}{P_2(\xi, \lambda)}\hat \bff,& 
&\ell_2(\xi, \lambda) \hat \bff = \frac{\lambda + |\xi|^2 + a}{P_2(\xi, \lambda)} \frac{\xi}{|\xi|^2} \xi \cdot \hat \bff, \\
&m_1(\xi, \lambda) \widehat{\DV\bg} = -\frac{\beta |\xi|^2}{P_2(\xi, \lambda)}\widehat{\DV\bg},&
&m_2(\xi, \lambda) \widehat{\DV\bg} = -\frac{\beta \xi}{P_2(\xi, \lambda)}\xi\cdot \widehat{\DV\bg},\\
&\tilde \ell_1(\xi, \lambda) \hat \bff 
= \frac{\beta}{P_2(\xi, \lambda)}(i\xi \otimes \hat \bff - \hat \bff \otimes \xi),& 
&\tilde \ell_2(\xi, \lambda) \hat \bff = \frac{2\beta}{P_2(\xi, \lambda)} \frac{\xi \otimes \xi}{|\xi|^2} \xi \cdot \hat \bff, \\
&\tilde m_1(\xi, \lambda) \widehat{\DV\bg} 
= \frac{\beta^2 |\xi|^2(\xi \otimes \widehat{\DV\bg} + \widehat{\DV\bg} \otimes \xi)}
{(\lambda + |\xi|^2 + a)P_2(\xi, \lambda)}, \\
&\tilde m_2(\xi, \lambda) \widehat{\DV\bg} 
= \frac{2\beta^2 \xi \otimes \xi}{(\lambda + |\xi|^2 + a)P_2(\xi, \lambda)} \xi \cdot \widehat{\DV\bg},&
&\tilde m_3(\xi, \lambda) \hat{\bg} = \frac{1}{\lambda + |\xi|^2 + a} \hat{\bg}.
\end{align*}
Here, we set the integral path 
$\Gamma = \Gamma^+ \cup \Gamma^-$ as follows:
\begin{equation}\label{path}
\Gamma^\pm
= \{\lambda \in \C \mid 
\lambda = \tilde \lambda_0 (\sigma) + s e^{\pm i(\pi-\sigma)},
s: 0 \to \infty \}
\end{equation}
for some $\sigma_0 < \sigma < \pi/2$ with
$\tilde{\lambda}_0 (\sigma) = 2\lambda_0(\sigma)/ \sin \sigma$,
where $\lambda_0 (\sigma)$ is the same as in Theorem \ref{thm:Rbdd}.
We then have the following $L_p$-$L_q$ decay estimates for $S(t)$ and $T(t)$.

\begin{thm}\label{decay for u}
Let $S(t)$ and $T(t)$ be the solution operators of \eqref{l3} 
such that $\bu = S(t)(\bff, \bg)$ and $\bQ = T(t)(\bff, \bg)$
 for $(\bff, \bg) \in X_q(\R^N)$.
Then, the following assertions hold.
\begin{enumerate}
\item\label{t>0}
If $t \geq 1$, $S(t)$ and $T(t)$ are decomposed as follows: 
\[
S(t)(\bff, \bg) = S_0(t)(\bff, \bg) + S_\infty(t)(\bff, \bg),
\enskip
T(t)(\bff, \bg) = T_0(t)(\bff, \bg) + T_\infty(t)(\bff, \bg)
\]
satisfying the following estimates.
\begin{enumerate}
\item\label{es low}
Let $1\leq q \leq 2 \leq p \leq \infty$.
\begin{align}
\|\pd^k_t S_0(t) (\bff, \bg)\|_{L_p(\R^N)}
&\leq C
t^{-\frac{N}{2} (\frac{1}{q} - \frac{1}{p}) - k}
\|(\bff, \bg)\|_{W^{0, 1}_q(\R^N)}
\enskip (q \neq 2), \label{esemi1} \\
\|\nabla^j S_0(t) (\bff, \bg)\|_{L_p(\R^N)}
&\leq C
t^{-\frac{N}{2} (\frac{1}{q} - \frac{1}{p}) - \frac{j}{2}}
\|(\bff, \bg)\|_{W^{0, 1}_q(\R^N)},
\label{esemi1'} \\
\|\nabla^{j+1} T_0(t) (\bff, \bg)\|_{L_p(\R^N)}
&\leq C
t^{-\frac{N}{2} (\frac{1}{q} - \frac{1}{p}) - \frac{j}{2}}
\|(\bff, \bg)\|_{W^{0, 1}_q(\R^N)} \label{est} 
\end{align}
with $j \in \N$, $k \in \N_0$ and 
some positive constant $C$ depending on 
$j$, $k$, $p$, and $q$.
In addition, 
\begin{equation}\label{t0}
\|\pd_t^k \nabla^j T_0 (t) (\bff, \bg)\|_{L_p(\R^N)}
\leq C
t^{-\frac{N}{2} (\frac{1}{q} - \frac{1}{p})- \frac{j}{2}-k}
\|(\bff, \bg)\|_{W^{0, 1}_q(\R^N)} \enskip (q \neq 2)
\end{equation}
with $j, k \in \N_0$ and  
some positive constant $C$ depending on 
$j$, $k$, $p$ and $q$.


\item\label{es high}
Let $1<p<\infty$.
\begin{align*}
&\|\pd_t (S_\infty(t)(\bff, \BG), T_\infty(t)(\bff, \BG))\|_{W^{0,1}_p(\BR^N)}
+ \|(S_\infty(t)(\bff, \BG), T_\infty(t)(\bff, \BG))\|_{W^{2,3}_p(\BR^N)}\\
&\quad\leq Ce^{-\gamma_\infty t}\|(\bff, \BG)\|_{W^{0,1}_p(\BR^N)}
\end{align*}
with  
some positive constants $C$ and $\gamma_\infty$. 

\end{enumerate}


\item \label{0<t<1}
Let $1 < q \leq p \leq \infty$ and 
$(p, q) \neq (\infty, \infty)$.
If $0 < t \leq 1$, $S(t) (\bff, \bg)$ and $T(t) (\bff, \bg)$ satisfy the following estimate:
\begin{equation}\label{t small}
\begin{aligned}
&\|\nabla^j (S(t) (\bff, \bg), T(t) (\bff, \bg))\|_{W^{0, 1}_p(\R^N)}
\leq C
t^{-\frac{N}{2} (\frac{1}{q} - \frac{1}{p}) - \frac{j}{2}}
\|(\bff, \bg)\|_{W^{0, 1}_q(\R^N)},\\
&\|\pd_t (S(t) (\bff, \bg), T(t) (\bff, \bg))\|_{W^{0, 1}_p(\R^N)}
\leq C
t^{-\frac{N}{2} (\frac{1}{q} - \frac{1}{p}) - 1}
\|(\bff, \bg)\|_{W^{0, 1}_q(\R^N)}
\end{aligned}
\end{equation}
with 
$j = 0, 1$ and some positive constant $C$ depending on 
$p$ and $q$.
Moreover,
\begin{equation}\label{esemi2}
\begin{aligned}
&\|\nabla^j (S(t) (\bff, \bg), T(t) (\bff, \bg))\|_{W^{0, 1}_q(\R^N)}
\leq C
\|(\bff, \bg)\|_{W^{j, j+1}_q(\R^N)}
\quad(j=0,1,2),\\
&\|\pd_t (S(t) (\bff, \bg), T(t) (\bff, \bg))\|_{W^{0, 1}_q(\R^N)}
\leq C
\|(\bff, \bg)\|_{W^{2, 3}_q(\R^N)}
\end{aligned}
\end{equation}
with some positive constant $C$.
\end{enumerate}
\end{thm}

We shall prove Theorem \ref{decay for u}.
For Theorem \ref{decay for u} \eqref{0<t<1},
by \eqref{resolvent es} and $0< t \leq 1$, 
we see that $S(t)(\bff, \bg)$ and $T(t)(\bff, \bg)$ satisfy
\begin{align*}
\|\nabla^j (S(t) (\bff, \bg), T(t) (\bff, \bg))\|_{W^{0, 1}_p(\R^N)}
&\leq C_\sigma e^{\tilde{\lambda}_0(\sigma) t} t^{-\frac{j}{2}} \|(\bff, \bg)\|_{W^{0, 1}_p(\R^N)} \\
&\leq C_\sigma t^{-\frac{j}{2}} \|(\bff, \bg)\|_{W^{0, 1}_p(\R^N)}
\end{align*}
with $j=0, 1, 2$.
Moreover, we have
\begin{align*}
\|\pd_t (S(t) (\bff, \bg), T(t) (\bff, \bg))\|_{W^{0, 1}_p(\R^N)}
&\leq C_\sigma t^{-1} \|(\bff, \bg)\|_{W^{0, 1}_p(\R^N)},
\end{align*}
which implies that \eqref{t small}.
Furthermore, \eqref{esemi2} follows from 
the fact that $\{e^{\CA t}\}_{t \geq 0}$
is continuous analytic semigroup.
Thus, we prove Theorem \ref{decay for u} \eqref{t>0}.

\subsubsection{Analysis of low frequency parts}\label{subsection low}
To prove Theorem \ref{decay for u} \eqref{t>0} (a),
we prepare several lemmas.
Recall that roots of $P_2(\xi, \lambda)$ have 
the following expansion formula:
\begin{equation}\label{expansion}
\lambda_+ = -|\xi|^2 + O(|\xi|^4), \enskip
\lambda_- = -|\xi|^2 - a + O(|\xi|^4).
\end{equation}
We set $\sigma_0 = \tan^{-1}\{(|\xi|^2/8)/|\xi|^2\} = \tan^{-1}1/8$
and
\begin{align*}
\Gamma^{\pm}_1 &=
\{\lambda \in \BC \mid \lambda = -|\xi|^2 + (|\xi|^2/4) e^{\pm is}, 
\enskip s: 0 \to \pi/2 \},\\
\Gamma^{\pm}_2 &=
\{\lambda \in \BC \mid \lambda = -(|\xi|^2(1-s) + \gamma_0 s)
\pm i 
\{(|\xi|^2/4)(1-s) + \tilde{\gamma}_0 s\}, \enskip s: 0 \to 1 \},\\
\Gamma^{\pm}_3 &=
\{\lambda \in \BC \mid \lambda = - (\gamma_0 \pm i\tilde{\gamma}_0)
+ s e^{\pm i(\pi - \sigma_0)}, \enskip s: 0 \to \infty \},
\end{align*}
where 
$\gamma_0 = \lambda_0 (\sigma_0)$ and
$\tilde{\gamma}_0 = (\lambda_0(\sigma) + \tilde{\lambda}_0(\sigma) )/8
= \gamma_0 (2\sqrt{65}+1)/8$.
By Cauchy's integral theorem,
$S_0(t)(\bff, \bg)$ and $T_0(t)(\bff, \bg)$ can be decomposed by 
\[
S_0(t)(\bff, \bg) = \sum^3_{k=1} S_0^k(t)(\bff, \bg),
\enskip
T_0(t)(\bff, \bg) = \sum^3_{k=1} T_0^k(t)(\bff, \bg),
\]
where
\begin{align*}
&S_0^k(t)(\bff, \bg)
=\frac{1}{2\pi i} \sum^2_{j = 1}
\CF^{-1} \left[\int_{\Gamma^+_k \cup \Gamma^-_k}
e^{\lambda t} 
\ell_j(\xi, \lambda) \vp_0 \hat \bff \,d\lambda \right](x)
+
\frac{1}{2\pi i} \sum^2_{j = 1}
\CF^{-1} \left[\int_{\Gamma^+_k \cup \Gamma^-_k}
e^{\lambda t} 
m_j(\xi, \lambda) \vp_0 \widehat{\DV\bg} \,d\lambda \right](x),\\
&T_0^k(t)(\bff, \bg)
=\frac{1}{2\pi i} \sum^2_{j = 1}
\CF^{-1} \left[\int_{\Gamma^+_k \cup \Gamma^-_k} e^{\lambda t} 
\tilde \ell_j(\xi, \lambda) \vp_0 \hat \bff \,d\lambda \right](x)
+ \frac{1}{2\pi i} \sum^2_{j = 1}\CF^{-1} \left[\int_{\Gamma^+_k \cup \Gamma^-_k} 
e^{\lambda t} 
\tilde m_j(\xi, \lambda) \vp_0 \widehat{\DV\bg} \,d\lambda \right](x)\\
&\enskip + \frac{1}{2\pi i} \CF^{-1} \left[\int_{\Gamma^+_k \cup \Gamma^-_k} 
e^{\lambda t} 
\tilde m_3 (\xi, \lambda) \vp_0 \hat{\bg} \,d\lambda \right](x)
+\frac{1}{2\pi i} \frac{a}{N}\CF^{-1} \left[\int_{\Gamma^+_k \cup \Gamma^-_k} 
e^{\lambda t} 
\tilde m_3 (\xi, \lambda) \vp_0{\rm tr}\,\hat \BQ \bI \,d\lambda \right](x).
\end{align*}
We only consider estimates on $\Gamma^+_k$ $( k = 1,2,3)$ below
because the case $\Gamma^-_k$ can be proved similarly.

\noindent
\underline{\bf Estimates on $\Gamma^+_1$.}
First, we consider the case $\Gamma^+_1$.

\begin{lem}\label{gamma1}
Let $1 \leq q \leq 2 \leq p \leq \infty$ and let $\bff \in L_q(\R^N)^N$.
Assume that there exist positive constants $A_1 \in (0, 1)$ and $C$
such that for any $|\xi| \in (0, A_1)$
\[
|\ell(\xi, \lambda)| \leq C |\xi|^{-2}, 
\]
Then, there exist positive constant $A_0 \in (0, 1)$ and $C$ such that 
for any $t>0$
\begin{align}
\left\|
\pd_t^k \nabla^j \CF^{-1} 
\left[ \int_{\Gamma^+_1} e^{\lambda t} \ell(\xi, \lambda) \vp_0 \hat{\bff} 
\,d\lambda \right]
\right\|_{L_p(\BR^N)}
&\leq C t^{-\frac{N}{2}(\frac{1}{q}-\frac{1}{p})-\frac{j}{2} - k}\|\bff\|_{L_q(\BR^N)}
\end{align}
with $j, k \in \N_0$.
\end{lem}

\begin{proof}
Choosing $A_0=A_1$, by
$L_p$-$L_q$ estimates of heat kernel,
and Parseval's theorem, we have 
\begin{equation*}
\begin{aligned}
&\left\|
\pd_t^k \nabla^j \CF^{-1} 
\left[ \int_{\Gamma^+_1} e^{\lambda t} \ell (\xi, \lambda) \vp_0 \hat{\bff} 
\,d\lambda \right]
\right\|_{L_p(\BR^N)}\\
& \enskip
\leq C
\left\|
\CF^{-1} 
\left[ \int^{\pi/2}_0 e^{-|\xi|^2t+\frac{|\xi|^2}{4}t e^{is}} 
\ell(\xi, \lambda) |\xi|^{2+j+2k} e^{is}
\left(-1+\frac{e^{is}}{4} \right)^k\, ds \vp_0 \hat{\bff} \right]
\right\|_{L_p(\BR^N)}\\
&\enskip
\leq C t^{-\frac{N}{2}(\frac{1}{2}-\frac{1}{p})}
\left\|
\CF^{-1} 
\left[ \int^{\pi/2}_0 e^{-\frac{|\xi|^2}{2} t+\frac{|\xi|^2}{4}t e^{is}} 
\ell(\xi, \lambda) |\xi|^{2+j+2k} e^{is} 
\left(-1+\frac{e^{is}}{4} \right)^k\, ds \vp_0 \hat{\bff} \right]
\right\|_{L_2(\BR^N)}\\
&\enskip
\leq C t^{-\frac{N}{2}(\frac{1}{2}-\frac{1}{p})-\frac{j}{2}-k}
\left\|
\CF^{-1} 
\left[e^{-\frac{|\xi|^2}{8}t} \hat{\bff} \right]
\right\|_{L_2(\BR^N)}\\
&\enskip
\leq C t^{-\frac{N}{2}(\frac{1}{q}-\frac{1}{p})-\frac{j}{2}-k}
\|\bff \|_{L_q(\BR^N)}.
\end{aligned}
\end{equation*}

\end{proof}

\noindent
\underline{\bf Estimates on $\Gamma^+_2$.}
Next, we consider the case $\Gamma^+_2$.
\begin{lem}\label{gamma2}
Let $1 \leq q \leq 2 \leq p \leq \infty$ and let $\bff \in L_q(\R^N)^N$.
Assume that there exist positive constants $A_1 \in (0, 1)$ and $C$
such that for any $|\xi| \in (0, A_1)$
\[
|\ell(\xi, \lambda)| \leq C (|\lambda| + |\xi|^2)^{-1}.
\]
Then, there exist positive constant $A_0 \in (0, A_1)$ and $C$ such that 
for any $t>0$
\begin{align}
\left\|
\pd_t^k \CF^{-1} 
\left[ \int_{\Gamma^\pm_2} e^{\lambda t} \ell(\xi, \lambda) \vp_0 \hat{\bff} 
\,d\lambda \right]
\right\|_{L_p(\BR^N)}
&\leq C t^{-\frac{N}{2}(\frac{1}{q}-\frac{1}{p})-k}\|\bff\|_{L_q(\BR^N)} \enskip(q \neq 2), \label{pd t}\\
\left\|
\nabla^j \CF^{-1} 
\left[ \int_{\Gamma^\pm_2} e^{\lambda t} \ell(\xi, \lambda) \vp_0 \hat{\bff} 
\,d\lambda \right]
\right\|_{L_p(\BR^N)}
&\leq C t^{-\frac{N}{2}(\frac{1}{q}-\frac{1}{p})-\frac{j}{2}}\|\bff\|_{L_q(\BR^N)} \label{nabla}
\end{align}
with $j \in \N$ and $k \in \N_0$.
\end{lem}

\begin{proof}
First, we prove \eqref{pd t}.
We choose $A_0 \in (0, A_1)$ in such a way that
\begin{equation}\label{a0}
0< \frac{\gamma_0 + \tilde \gamma_0 + |\xi|^2}{\gamma_0 - |\xi|^2} <C
\end{equation}
for any $|\xi|\in (0, A_0)$ with some positive constant $C$.
Since $|\xi|$ and $\tilde\gamma_0$ satisfy 
$|\xi|^2<1< (2\sqrt{65}+1)/8 < \tilde\gamma_0$, 
we have
\begin{equation}\label{spectrum2}
|\lambda + |\xi|^2| \geq C \{|\xi|^2 + |\xi|^2(1-s) + \gamma_0 s\}.
\end{equation}
In fact,
\begin{align*}
|\lambda + |\xi|^2|^2&= \left|s(|\xi|^2-\gamma_0)
+i\left\{\frac{|\xi|^2}{4}(1-s) + \tilde\gamma_0 s\right\}\right|^2\\
&=s^2(\gamma_0-|\xi|^2)^2 + \left\{\frac{|\xi|^2}{4}(1-s) + \tilde\gamma_0 s\right\}^2 \\
&\geq \left\{\frac{|\xi|^2}{4}(1-s) + \tilde\gamma_0 s\right\}^2
\end{align*}
yields that
\begin{align*}
|\lambda + |\xi|^2|
&\geq \frac{|\xi|^2}{4}(1-s) + \tilde\gamma_0 s\\
&= \frac12\left\{ \frac{|\xi|^2}{4}(1-s) + \tilde\gamma_0 s\right\} 
+ \frac12\left\{ \frac{|\xi|^2}{4}(1-s) 
+ \tilde\gamma_0 s\right\}\\
&\geq \frac18|\xi|^2 + \frac s8(4\tilde\gamma_0 - |\xi|^2) 
+ \frac18(|\xi|^2(1-s) + \tilde\gamma_0 s)
\\
&\geq  \frac18(|\xi|^2 + |\xi|^2(1-s) + \tilde\gamma_0 s).
\end{align*}
By \eqref{spectrum2}, $L_p$-$L_q$ estimates of heat kernel, and
Parseval's theorem, 
we have 
\begin{equation*}
\begin{aligned}
&\left\|
\pd_t^k \CF^{-1} 
\left[ \int_{\Gamma^+_2} e^{\lambda t} \ell(\xi, \lambda) \vp_0 \hat{\bff} 
\,d\lambda \right]
\right\|_{L_p(\BR^N)}\\
&\enskip
\leq C t^{-\frac{N}{2}(\frac{1}{2}-\frac{1}{p})}
\left\|
 e^{-\frac{|\xi|^2}{2}t} \int^1_0 e^{-\frac{1}{4}(|\xi|^2(1-s) + \gamma_0 s) t} 
\frac{(\gamma_0 + \tilde \gamma_0 + |\xi|^2)(|\xi|^2(1-s) + \gamma_0 s)^k}{|\xi|^2 + |\xi|^2(1-s) + \gamma_0 s} \, ds \vp_0  \hat{\bff} 
\right\|_{L_2(\BR^N)}\\
&\enskip
\leq C t^{-\frac{N}{2}(\frac{1}{2}-\frac{1}{p})}
\left\|
e^{-\frac{|\xi|^2}{2}t} \int^1_0 e^{-\frac{1}{4}(|\xi|^2(1-s) + \gamma_0 s)t} 
(|\xi|^2(1-s) + \gamma_0 s )^{k-1+\frac{\delta}{2}} \, ds
(\gamma_0 + \tilde \gamma_0 + |\xi|^2) |\xi|^{-\delta} \vp_0 \hat{\bff} 
\right\|_{L_2(\BR^N)}
\end{aligned}
\end{equation*}
for a sufficiently small $\delta >0$.
Here, by change of variables, we have
\[
\int^1_0 e^{-\frac{1}{4}(|\xi|^2(1-s) + \gamma_0 s)t} 
(|\xi|^2(1-s) + \gamma_0 s )^{k-1+\frac{\delta}{2}} \, ds
\leq
C t^{-k-\frac{\delta}{2}}
\frac{1}{\gamma_0 - |\xi|^2},
\]
which combined with \eqref{a0} and Young's inequality with $1+1/2 = 1/r + 1/q$ for $1 \leq q <2$, 
we have
\begin{equation*}
\begin{aligned}
&\left\|
\pd_t^k \CF^{-1} 
\left[ \int_{\Gamma^+_2} e^{\lambda t} \ell(\xi, \lambda) \vp_0 \hat{\bff} 
\,d\lambda \right]
\right\|_{L_p(\BR^N)}\\
&\enskip
\leq C t^{-\frac{N}{2}(\frac{1}{2}-\frac{1}{p})-k}
t^{-\frac{\delta}{2}}
\left\|
e^{-\frac{|\xi|^2}{2}t} 
\frac{\gamma_0 + \tilde \gamma_0 + |\xi|^2}{\gamma_0 - |\xi|^2}
|\xi|^{-\delta} \vp_0 \hat{\bff} 
\right\|_{L_2(\BR^N)}\\
&\enskip
\leq Ct^{-\frac{N}{2}(\frac{1}{2}-\frac{1}{p})-k}
t^{-\frac{\delta}{2}}
\left\|
e^{-\frac{|\xi|^2}{2}t} 
|\xi|^{-\delta}\hat{\bff} 
\right\|_{L_2(\BR^N)}\\
&\enskip
\leq C t^{-\frac{N}{2}(\frac{1}{2}-\frac{1}{p})-k}
t^{-\frac{\delta}{2}}
\left\|
\CF^{-1} 
\left[ 
e^{-\frac{|\xi|^2}{2}t} 
|\xi|^{-\delta}
\right]
\right\|_{L_r(\R^N)}
\|\bff \|_{L_q(\BR^N)}.
\end{aligned}
\end{equation*}
Here,
\[
\left\|
\CF^{-1} 
\left[ 
e^{-\frac{|\xi|^2}{2}t} 
|\xi|^{-\delta}
\right]
\right\|_{L_r(\R^N)}
\leq
C t^{\frac{\delta}{2} - \frac{N}{2}\left(\frac{1}{q} - \frac{1}{2}\right)}
\]
provided by $r>1$, so that we have
\begin{equation*}
\begin{aligned}
&\left\|
\pd_t^k \CF^{-1} 
\left[ \int_{\Gamma^+_2} e^{\lambda t} \ell(\xi, \lambda) \vp_0 \hat{\bff} 
\,d\lambda \right]
\right\|_{L_p(\BR^N)}\\
&\enskip
\leq C t^{-\frac{N}{2}(\frac{1}{2}-\frac{1}{p})-k}
t^{\frac{\delta}{2} - \frac{N}{2}\left(\frac{1}{q} - \frac{1}{2}\right)}
t^{-\frac{\delta}{2}}
\|
\bff \|_{L_q(\BR^N)}\\
&\enskip
= Ct^{-\frac{N}{2}(\frac{1}{q}-\frac{1}{p})-k}
\|\bff \|_{L_q(\BR^N)}.
\end{aligned}
\end{equation*}

By a similar calculation, we can prove \eqref{nabla}, but without using Young's inequality.
In fact, for $j \in \N$,
\begin{align*}
&\left\|
\nabla^j \CF^{-1} 
\left[ \int_{\Gamma^+_2} e^{\lambda t} \ell(\xi, \lambda) \vp_0 \hat{\bff} 
\,d\lambda \right]
\right\|_{L_p(\BR^N)}\\
&\enskip
\leq C t^{-\frac{N}{2}(\frac{1}{2}-\frac{1}{p})}
\left\|
e^{-\frac{|\xi|^2}{2}t} \int^1_0 e^{-\frac{1}{4}(|\xi|^2(1-s) + \gamma_0 s)t} 
(|\xi|^2(1-s) + \gamma_0 s )^{-1+\frac{\delta}{2}} \, ds
(\gamma_0 + \tilde \gamma_0 + |\xi|^2) |\xi|^{j-\delta} \vp_0 \hat{\bff} 
\right\|_{L_2(\BR^N)}\\
&\enskip
\leq C t^{-\frac{N}{2}(\frac{1}{2}-\frac{1}{p})}
t^{-\frac{\delta}{2}}
\left\|
e^{-\frac{|\xi|^2}{2}t} 
\frac{\gamma_0 + \tilde \gamma_0 + |\xi|^2}{\gamma_0 - |\xi|^2}
|\xi|^{j-\delta} \vp_0 \hat{\bff} 
\right\|_{L_2(\BR^N)}\\
&\enskip
\leq C t^{-\frac{N}{2}(\frac{1}{q}-\frac{1}{p})- \frac{j}{2}}
\|
\bff \|_{L_q(\BR^N)}.
\end{align*}

\end{proof}

\noindent
\underline{\bf Estimates on $\Gamma^+_3$.}
Finally, we consider the case $\Gamma^+_3$.
\begin{lem}\label{gamma3}
Let $1 \leq q \leq 2 \leq p \leq \infty$ and let $\bff \in L_q(\R^N)^N$.
Assume that there exist positive constants $A_1 \in (0, 1)$ and $C$
such that for any $|\xi| \in (0, A_1)$
\[
|\ell (\xi, \lambda)| \leq C |\lambda| ^{-1}.
\]
Then, there exist positive constant $A_0 \in (0, A_1)$ and $C$ such that 
for any $t>0$
\begin{align}\label{es:gamma3}
\left\|
\pd_t^k \nabla^j \CF^{-1} 
\left[ \int_{\Gamma^\pm_3} e^{\lambda t} \ell (\xi, \lambda) \vp_0 \hat{\bff} 
\,d\lambda \right]
\right\|_{L_p(\BR^N)}
&\leq C t^{-\frac{N}{2}(\frac{1}{q}-\frac{1}{p})-\frac{j}{2}-k}\|\bff\|_{L_q(\BR^N)}
\end{align}
with $j, k \in \N_0$.
\end{lem}

\begin{proof}
We choose $A_0 \in (0, A_1)$ in such a way that
$2|\xi|^2 \leq \gamma_0/2$
for any $|\xi| \in (0, A_0)$. 
By $L_p$-$L_q$ estimates of heat kernel,
and Parseval's theorem, 
we have 
\begin{equation*}
\begin{aligned}
&\left\|
\pd_t^k \nabla^j \CF^{-1} 
\left[ \int_{\Gamma^+_3} e^{\lambda t} \ell (\xi, \lambda) \vp_0 \hat{\bff} 
\,d\lambda \right]
\right\|_{L_p(\BR^N)}\\
&\enskip
\leq C t^{-\frac{N}{2}(\frac{1}{2}-\frac{1}{p})}
\left\|
 e^{-|\xi|^2t} \int^\infty_0 
e^{-\left(\frac{\gamma_0}{2} + s\cos \sigma_0 \right)t} 
\frac{1}{|\lambda|^{1-k}} \, ds  |\xi|^j \vp_0 \hat{\bff} 
\right\|_{L_2(\BR^N)}\\
&\enskip
\leq C t^{-\frac{N}{2}(\frac{1}{q}-\frac{1}{p})-\frac{j}{2}}
\int^\infty_0 \frac{e^{-\left(\frac{\gamma_0}{2} + s\cos \sigma_0 \right)t}}{|\lambda|^{1-k}}\,ds
\|\bff\|_{L_q(\BR^N)}\\
&\enskip
\leq C \left\{
\begin{aligned}
&(|\log t| + 1) e^{-\frac{\gamma_0}{2}t} \|\bff\|_{L_q(\BR^N)}
&(k=0), \\
&t^{-k} e^{-\frac{\gamma_0}{2}t} \|\bff\|_{L_q(\BR^N)}
&(k>0).
\end{aligned}
\right.
\end{aligned}
\end{equation*}
Here, we used the following estimate
\begin{equation}\label{log}
\int^\infty_0 \frac{e^{({\rm Re} \lambda) t}}{|\lambda|^{1-k}}\,ds
\leq C
 \left\{
\begin{aligned}
&(|\log t| + 1) e^{-\gamma_0 t}&(k=0), \\
&t^{-k} e^{-\gamma_0 t}&(k>0).
\end{aligned}
\right.
\end{equation}
Thus, we have \eqref{es:gamma3}.
\end{proof}

\begin{proof}[proof of Theorem \ref{decay for u} \eqref{t>0} (a)]
To obtain the estimates of $S^1_0(t)$, 
we first consider estimate of $P_2(\xi, \lambda)$
for $\lambda \in \Gamma^+_1$.
By \eqref{expansion}, there exists positive constant $A_1 \in (0, 1)$ and $C$
such that for any $\lambda \in \Gamma^+_1$ and $|\xi| < A_1$
\[
P_2(\xi, \lambda)
= \frac{|\xi|^2}{4} \left(\frac{|\xi|^2}{4}e^{is} + a\right).
\]
By
$|(|\xi|^2/4)e^{is} + a| \geq C(|\xi|^2 + a)$ for $s \in [0, \pi/2]$,
we have
\begin{equation}\label{p1p2}
|P_2(\xi, \lambda)^{-1}|
\leq C(|\xi|^4 + a |\xi|^2)^{-1}.
\end{equation}
Using $|\lambda+|\xi|^2+a|\leq C(|\xi|^2 + a)$ and \eqref{p1p2},
we have
\begin{equation}\label{esgamma1}
|\ell_j(\xi, \lambda)| \leq C|\xi|^{-2},
\enskip
|m_j(\xi, \lambda)| \leq C|\xi|^{-2}
\end{equation}
for $\lambda \in \Gamma^+_1$ and $j = 1, 2$.

We next consider the estimate of $S_0^2(t)$ and $S_0^3(t)$.
Noting that $\Gamma^+_2, \Gamma^+_3 \subset \Sigma_{\sigma_0, 0} $,
by \eqref{spectrum1} and \eqref{spectrum3},
we know that
\begin{equation} \label{esgamma23m}
\begin{aligned}
&|m_j(\xi, \lambda)| 
\leq C (|\lambda| + |\xi|^2)^{-1} &\text{ for } &\lambda \in \Gamma^+_2,\\
&|m_j(\xi, \lambda)| \leq C |\lambda|^{-1} &\text{ for } &\lambda \in \Gamma^+_3 
\end{aligned}
\end{equation}
with $j=1, 2$. 
Moreover, by the estimate
\[
|P_2(\xi, \lambda)^{-1}| \leq C(|\lambda| + |\xi|^2)^{-1} (|\lambda| + |\xi|^2 + a)^{-1},
\]
which follows from \eqref{spectrum1} and the expansion formula \eqref{p2 low},
we have
\begin{equation} \label{esgamma23l}
\begin{aligned}
&|\ell_j(\xi, \lambda)| 
\leq C (|\lambda| + |\xi|^2)^{-1} &\text{ for } &\lambda \in \Gamma^+_2,\\
&|\ell_j(\xi, \lambda)| \leq C |\lambda|^{-1} &\text{ for } &\lambda \in \Gamma^+_3 
\end{aligned}
\end{equation}
with $j=1, 2$. 

By \eqref{esgamma1}, \eqref{esgamma23m}, and \eqref{esgamma23l},
we can apply Lemma \ref{gamma1}, Lemma \ref{gamma2}, and Lemma \ref{gamma3},
which gives \eqref{esemi1} and \eqref{esemi1'}.

Similarly, we have the estimate of $\nabla^{k+1} T_0(t) (\bff, \bg)$
for $k \in \N$.
In fact, $i\xi_\ell \tilde \ell_j (\xi, \lambda)$, 
$i\xi_\ell \tilde m_j (\xi, \lambda)$, and 
$\tilde m_3 (\xi, \lambda)$
satisfy the assumption of Lemma \ref{gamma1}, Lemma \ref{gamma2}, and Lemma \ref{gamma3}
for $\ell = 1, \dots, N$ and $j = 1, 2$.
Using \eqref{esd}, we have \eqref{est}.

We finally consider the estimate of $\pd_t^k \nabla^j T_0(t) (\bff, \bg)$
for $j, k \in \N_0$.
$T_0(t) (\bff, \bg)$ can be rewritten by the following form:
\begin{align*}
T_0(t)(\bff, \bg)
&=\frac{\beta}{2\pi i} 
\CF^{-1} \left[\int_\Gamma e^{\lambda t} 
\tilde m_3(\xi, \lambda) \vp_0 \widehat{\bD(\bu)} \,d\lambda \right](x)
+ \frac{1}{2\pi i} \sum^2_{j = 1}\CF^{-1} \left[\int_\Gamma e^{\lambda t} 
\tilde m_3(\xi, \lambda) \vp_0 \hat\bg \,d\lambda \right](x)\\
&\enskip
+\frac{1}{2\pi i} \frac{a}{N}\CF^{-1} \left[\int_\Gamma e^{\lambda t} 
\tilde m_3(\xi, \lambda) \vp_0 \tr \hat{\bQ} \bI \,d\lambda \right](x)\\
&=: I_0(t) + II_0(t) + III_0(t).
\end{align*}
Recalling that $\tilde m_3(\xi, \lambda)$ satisfies Lemma \ref{gamma1},
Lemma \ref{gamma2}, and Lemma \ref{gamma3} and using \eqref{esd},
we have
\begin{equation}\label{23}
\begin{aligned}
\|\pd_t^k \nabla^j II_0\|_{L_p(\R^N)} + \|\pd_t^k \nabla^j III_0\|_{L_p(\R^N)} 
&\leq Ct^{-\frac{N}{2} (\frac{1}{q} - \frac{1}{p})-\frac{j}{2}-k}
(\|\bg\|_{L_q(\R^N)} + \|d\|_{L_q(\R^N)})\\
&\leq Ct^{-\frac{N}{2} (\frac{1}{q} - \frac{1}{p})-\frac{j}{2}-k}
\|\bg\|_{L_q(\R^N)}
\end{aligned}
\end{equation}
for $q \neq 2$, $j, k\in \N_0$.
To get $L_p$-$L_2$ decay estimates for $I_0$, 
we reconsider the estimates for $\lambda \in \Gamma_2$ as follows.
By the same calculation as in the proof of \eqref{nabla} in Lemma \ref{gamma2},
we have
\begin{equation}\label{du}
\begin{aligned}
&\left\|\pd_t^k \nabla^j \CF^{-1} \left[\int_{\Gamma^+_2} e^{\lambda t} 
\tilde m_3(\xi, \lambda) \vp_0 \widehat{\bD(\bu)} \,d\lambda \right]\right\|_{L_p(\R^N)}\\
&\enskip 
\leq C t^{-\frac{N}{2}(\frac{1}{2}-\frac{1}{p})-\frac{j}{2}}\\
&\enskip \times
\left\|
e^{-\frac{|\xi|^2}{8} t} \int^1_0 e^{-\frac{1}{4}(|\xi|^2(1-s) + \gamma_0 s)t} 
\frac{(|\xi|^2(1-s) + \gamma_0 s)^k}{|\xi|^2 + |\xi|^2(1-s) + \gamma_0 s } \, ds
(\gamma_0 + \tilde \gamma_0 + |\xi|^2)|\xi| \vp_0 \hat{\bu} 
\right\|_{L_2(\BR^N)}\\
&\enskip 
\leq C t^{-\frac{N}{2}(\frac{1}{q}-\frac{1}{p})-\frac{1}{2}-\frac{j}{2}-k}
\|S_0(t)(\bff, \bg)\|_{L_q(\BR^N)},
\end{aligned}
\end{equation}
where $1 \leq q \leq 2 \leq p \leq \infty$.
Combining \eqref{du} with $(p, q) = (p, 2)$ and \eqref{esemi1} with $(p, q) = (2, q)$,
we have
\begin{equation}\label{1}
\|\pd_t^k \nabla^j I_0\|_{L_p(\R^N)}
\leq C t^{-\frac{N}{2}(\frac{1}{q}-\frac{1}{p})-\frac{1}{2}-\frac{j}{2}-k}
\|(\bff, \bg)\|_{W^{0, 1}_q(\BR^N)},
\end{equation}
where $1 \leq q \leq 2 \leq p \leq \infty$ and $q \neq 2$.
By \eqref{23} and \eqref{1}, we have \eqref{t0} for $t \geq 1$.
\end{proof}

\subsubsection{Analysis of high frequency parts}
We shall prove Theorem \ref{decay for u} \eqref{t>0} (b).
Recall that roots of $P_2(\xi, \lambda)$ have 
the following expansion formula:
\begin{equation}\label{expansion2}
\lambda_\pm = -(1 \pm i|\beta|)|\xi|^2 + O(1).
\end{equation}
Set $\gamma_\infty$ as follows:
\[
0 < \gamma_\infty \leq 2^{-1} (A_0/6)^2,
\]
where $A_0$ is the same as in subsection \ref{subsection low}.
Moreover, we set integral paths:
\begin{align*}
\Gamma^{\pm}_4 &=
\{\lambda \in \BC \mid \lambda = -\gamma_\infty \pm 
i \tilde \gamma_\infty  + s e^{\pm i(\pi - \sigma_0)}, 
\enskip s: 0 \to \infty \},\\
\Gamma^{\pm}_5 &=
\{\lambda \in \BC \mid \lambda = - \gamma_\infty \pm is, 
\enskip s: 0 \to \tilde\gamma_\infty \},
\end{align*}
where 
$\tilde \gamma_\infty = (\tilde \lambda_0(\sigma_0) + \gamma_\infty) \tan \sigma_0$,
$\tilde\lambda_0(\sigma_0) = 2\lambda_0(\sigma_0)/ \sin \sigma_0$,
and $\sigma_0$ is defined in \eqref{sigma}.
By Cauchy's integral theorem,
$S_\infty(t)(\bff, \bg)$ and $T_\infty(t)(\bff, \bg)$can be decomposed by 
\[
S_\infty(t)(\bff, \bg) = \sum^5_{k=4} S_\infty^k(t)(\bff, \bg),
\enskip
T_\infty(t)(\bff, \bg) = \sum^5_{k=4} T_\infty^k(t)(\bff, \bg),
\]
where
\begin{align*}
S_\infty^k(t)(\bff, \bg)
&=\frac{1}{2\pi i} 
\int_{\Gamma^+_k \cup \Gamma^-_k}
e^{\lambda t} \bu_\infty (\lambda, x) \,d\lambda,\\
\bu_\infty (\lambda, x) 
&=\CF^{-1} 
\left[\sum^2_{j = 1} \ell_j(\xi, \lambda) \vp_\infty \hat \bff \right](x)
+\CF^{-1} 
\left[\sum^2_{j = 1} m_j(\xi, \lambda) \vp_\infty \hat \bg \right](x),\\
T_\infty^k(t)(\bff, \bg)
&=\frac{1}{2\pi i} 
\int_{\Gamma^+_k \cup \Gamma^-_k}
e^{\lambda t} \bQ_\infty (\lambda, x) \,d\lambda,\\
\bQ_\infty (\lambda, x) 
&=\CF^{-1} 
\left[\sum^2_{j = 1} \tilde \ell_j(\xi, \lambda) \vp_\infty \hat \bff \right](x)
+\CF^{-1} 
\left[\sum^4_{j = 1} \tilde m_j(\xi, \lambda) \vp_\infty \hat \bg \right](x),\\
&\tilde m_4(\xi, \lambda) \hat \bg = \frac{1}{P_1(\xi, \lambda)} \tr \hat \bg \bI.
\end{align*}
We only consider estimates on $\Gamma^+_k$ $( k = 4, 5)$ below
because the case $\Gamma^-_k$ can be proved similarly.

\noindent
\underline{\bf Estimates on $\Gamma^+_4$.}
First, we consider the case $\Gamma^+_4$.
Let $1< q< \infty$.
Noting that $\Gamma^+_4 \subset \Sigma_{\sigma_0, \lambda_0(\sigma_0)}$,
by \eqref{rbdd u} \eqref{rbdd q} in Theorem \ref{thm:Rbdd}, there exists a positive constant $C$
such that for any $\lambda \in \Gamma^+_4$ and $(\bff, \bg) \in W^{0, 1}_q(\R^N)^N$, 
\begin{equation}\label{resolvent}
\begin{aligned}
&\|(\lambda \bu_\infty, \lambda \bQ_\infty)\|_{W^{0, 1}_q(\R^N)}
+\|(\lambda^{1/2}\nabla \bu_\infty, \lambda^{1/2}\nabla \bQ_\infty)\|_{W^{0, 1}_q(\R^N)}
+\|(\nabla^2 \bu_\infty, \nabla^2 \bQ_\infty)\|_{W^{0, 1}_q(\R^N)}\\
&\enskip \leq C\|(\bff, \bg)\|_{W^{0, 1}_q(\R^N)}.
\end{aligned}
\end{equation}
By \eqref{resolvent} and \eqref{log},
there exists a positive constant $C$ such that
for any $t>0$
\begin{equation}\label{s4}
\begin{aligned}
&\|(S^4_\infty(t) (\bff, \bg), T^4_\infty(t) (\bff, \bg))\|_{W^{0, 1}_q(\R^N)}
\leq C(|\log t|+1) e^{-\gamma_\infty t}\|(\bff, \bg)\|_{W^{0, 1}_q(\R^N)},\\
&\|(\nabla S^4_\infty(t) (\bff, \bg), \nabla T^4_\infty(t) (\bff, \bg))\|_{W^{0, 1}_q(\R^N)}
\leq Ct^{-\frac{1}{2}} e^{-\gamma_\infty t}\|(\bff, \bg)\|_{W^{0, 1}_q(\R^N)},\\
&\|((\pd_t, \nabla^2) S^4_\infty(t) (\bff, \bg), (\pd_t, \nabla^2) T^4_\infty(t) (\bff, \bg))\|_{W^{0, 1}_q(\R^N)}
\leq Ct^{-1} e^{-\gamma_\infty t}\|(\bff, \bg)\|_{W^{0, 1}_q(\R^N)}.
\end{aligned}
\end{equation}

\noindent
\underline{\bf Estimates on $\Gamma^+_5$.}
Next, we consider the case $\Gamma^+_5$.
Let $1< q< \infty$.
We shall estimate $\bv_\infty$ by Fourier multiplier theorem.
We consider estimates of $\lambda + |\xi|^2$, $P_1(\xi, \lambda)$, and $P_2(\xi, \lambda)$
for $\lambda \in \Gamma^+_5$ and $|\xi| \geq A_0/6$.

\begin{lem}
Let $\alpha \in \N_0^N$. There exists a positive constant $C$
such that for any $\lambda \in \Gamma^+_5$ and $|\xi| \geq A_0/6$,
\begin{equation}\label{es p2}\begin{aligned}
&|\pd_\xi^\alpha (\lambda + |\xi|^2)^{-1}|\leq C|\xi|^{-2-|\alpha|}, 
\\
&|\pd_\xi^\alpha (P_1(\xi, \lambda)^{-1})|
\leq C|\xi|^{-4-|\alpha|},
\\
&|\pd_\xi^\alpha (P_2(\xi, \lambda)^{-1})| 
\leq C|\xi|^{-4-|\alpha|}.
\end{aligned}
\end{equation}
\end{lem}

\begin{proof}
For any $\lambda \in \Gamma^+_5$ and $|\xi| \geq A_0/6$, we have
\begin{equation}\label{lambda xi 2}
|\lambda + |\xi|^2| 
\geq |{\rm Re}(\lambda + |\xi|^2)|
=|\xi|^2 - \gamma_\infty \geq |\xi|^2 - 2^{-1}(A_0/6)^2 
\geq 2^{-1}|\xi|^2.
\end{equation}
Similarly, we have
\begin{equation}\label{lambda xi a}
|\lambda + |\xi|^2 + a| 
\geq 2^{-1}|\xi|^2.
\end{equation}
 
By \eqref{expansion2}, we have
\begin{equation}\label{es p2 2}
|\lambda - \lambda_\pm| 
\geq
- \gamma_\infty + |\xi|^2
\geq 2^{-1}|\xi|^2
\end{equation}
for any $\lambda \in \Gamma^+_5$ and $|\xi| \geq A_0/6$.
By \eqref{lambda xi 2}, \eqref{lambda xi a}, 
\eqref{es p2 2}, $|\pd_\xi^\alpha |\xi|^2| \leq 2|\xi|^{2-|\alpha|}$, 
Leibniz's rule, and Bell's formula, we have
\eqref{es p2}.

\end{proof}
Let $\alpha \in \N_0^N$.
Note that there exists a positive constant $C$ 
such that for any
$\lambda \in \Gamma^+_5$ and $|\xi| \geq A_0/6$,
$|\lambda| \leq C |\xi|^2$.
Using \eqref{es p2} and Leibniz's rule,
there exists a positive constant $C$ such that for any
$\lambda \in \Gamma^+_5$ and $|\xi| \geq A_0/6$,
\begin{align*}
&|\pd_\xi^\alpha (\lambda L(\xi, \lambda), L(\xi, \lambda), 
i\xi_a L(\xi, \lambda), 
i\xi_a i\xi_b L(\xi, \lambda))|
\leq C|\xi|^{-|\alpha|},\\
&|\pd_\xi^\alpha (i\xi_a \lambda \tilde \ell_j (\xi, \lambda), i\xi_a \lambda \tilde m_k (\xi, \lambda), 
i\xi_a i\xi_b i\xi_c \tilde \ell_j (\xi, \lambda), i\xi_a i\xi_b i\xi_c \tilde m_k (\xi, \lambda))|
\leq C|\xi|^{-|\alpha|},
\end{align*}
where $L(\xi, \lambda) = \ell_j(\xi, \lambda), \tilde \ell_j(\xi, \lambda), m_j(\xi, \lambda), \tilde m_k(\xi, \lambda)$,
$j = 1, 2$, $k=1, \dots, 4$, $a, b, c = 1, \dots ,N$.
Applying Fourier multiplier theorem,
for any $(\bff, \bg) \in W^{0, 1}_q(\R^N)$,
we have
\[
\|(\lambda \bu_\infty, \lambda \bQ_\infty)\|_{W^{0, 1}_q(\R^N)}+
\|(\bu_\infty, \bQ_\infty)\|_{W^{2, 3}_q(\R^N)} \leq C \|(\bff, \bg)\|_{W^{0, 1}_q(\R^N)},
\]
so that
\begin{equation*}
\begin{aligned}
&\|\pd_t (S^5_\infty(t) (\bff, \bg), T^5_\infty(t) (\bff, \bg))\|_{W^{0, 1}_q(\R^N)}
+\|(S^5_\infty(t) (\bff, \bg), T^5_\infty(t) (\bff, \bg))\|_{W^{2, 3}_q(\R^N)}\\
&\enskip \leq C e^{-\gamma_\infty t}\|(\bff, \bg)\|_{W^{0, 1}_q(\R^N)}
\end{aligned}
\end{equation*}
for any $t>0$ with some constant $C$,
which combined with \eqref{s4}, we obtain Theorem \ref{decay for u}
\eqref{t>0} (b).

\section{A proof of Theorem \ref{global}}\label{proof of main}
We prove Theorem \ref{global} by the Banach fixed point argument.
Let $p$, $q_1$, and $q_2$ be exponents given in Theorem \ref{global}.
Let $\epsilon$ be a small positive number 
and let $\CN (\bu, \bQ)$ be the norm defined in \eqref{N}.
We define the underlying space $\CI_{T, \epsilon}$ as
\begin{align}
\CI_{T, \epsilon}
= \{ (\bu, \bQ) \in X_{p, q_2, T} \cap X_{p, q_1, T} 
\mid &(\bu, \bQ)|_{t=0} = (\bu_0, \bQ_0), 
\enskip \CN(\bu, \bQ) (T) \leq \epsilon, \nonumber\\
\label{infty}
&
\sup_{0<t<T}\|\bQ(\cdot, t)\|_{L_\infty(\R^N)}
\leq 1
\}.
\end{align}
Given $(\bu, \bQ) \in \CI_{T, \epsilon}$, 
let $(\bv, \bP)$ be a solution to the equation:
\begin{equation}\label{nonlinear1}
\left\{
\begin{aligned}
&\pd_t \bv -\Delta \bv + \nabla \fp +\beta \DV \Delta \bP = \bff(\bu, \bQ), 
\enskip \dv \bv=0
& \quad&\text{in $\R^N$ for $t \in (0, T)$}, \\
&\pd_t \bP - \beta \bD(\bv) -\Delta \bP + a\left(\bP - \frac{1}{N} \tr \bP \bI \right)
=\bg(\bu, \bQ) & \quad&\text{in $\R^N$ for $t \in (0, T)$}, \\
&(\bv, \bP) = (\bu_0, \bQ_0)& \quad&\text{in $\R^N$}.
\end{aligned}\right.
\end{equation}
We shall prove the following inequality in several steps:
\begin{equation}\label{extend}
\CN(\bv, \bP) (T) \leq C \epsilon^2.
\end{equation}

To prove \eqref{extend},
we divide $(\bv, \bP)$ into two parts : $\bv = \bv_1 + \bv_2$
and $\bP = \bP_1 + \bP_2$,
where $(\bv_1, \bP_1)$ satisfies time shifted equations:
\begin{equation}\label{shift}
\left\{
\begin{aligned}
&\pd_t \bv_1 +\lambda_1 \bv_1- \Delta \bv_1 + \nabla \fp +\beta \DV \Delta \bP_1 = \bff(\bu, \bQ), 
\quad\dv \bv_1=0
& \quad&\text{in $\R^N$ for $t \in (0, T)$}, \\
&\pd_t \bP_1 +\lambda_1 \bP_1 - \beta \bD(\bv_1) 
-\Delta \bP_1 + a\left(\bP_1 - \frac{1}{N} \tr \bP_1 \bI \right)
=\bg(\bu, \bQ) & \quad&\text{in $\R^N$ for $t \in (0, T)$}, \\
&(\bv_1, \bP_1) = (0, O)& \quad&\text{in $\R^N$}
\end{aligned}\right.
\end{equation}
and $(\bv_2, \bP_2)$ satisfies compensation equations:
\begin{equation}\label{compensation}
\left\{
\begin{aligned}
&\pd_t \bv_2 - P\Delta \bv_2 + \beta P\DV \Delta \bP_2 = \lambda_1\bv_1, 
& \quad&\text{in $\R^N$ for $t \in (0, T)$}, \\
&\pd_t \bP_2 - \beta \bD(\bv_2) 
-\Delta \bP_2 + a\left(\bP_2 - \frac{1}{N} \tr \bP_2 \bI \right)
=\lambda_1 \bP_1 & \quad&\text{in $\R^N$ for $t \in (0, T)$}, \\
&(\bv_2, \bP_2) =(\bu_0, \bQ_0) & \quad&\text{in $\R^N$},
\end{aligned}\right.
\end{equation}
where $P$ is solenoidal projection.

\subsection{Analysis of time shifted equations}
We consider 
the following linearized problem for \eqref{shift}:
\begin{equation}\label{linearshift}
\left\{
\begin{aligned}
&\pd_t \bu +\lambda_1 \bu- \Delta \bu + \nabla\fp + \beta \DV \Delta \bQ = \bff, 
\quad \dv\bu=0
& \quad&\text{in $\R^N$ for $t \in (0, T)$}, \\
&\pd_t \bQ +\lambda_1 \bQ - \beta \bD(\bu) 
-\Delta \bQ + a\left(\bQ - \frac{1}{N} \tr \bQ \bI \right)
=\bg & \quad&\text{in $\R^N$ for $t \in (0, T)$}, \\
&(\bu, \bQ) =(0, O) &\quad&\text{in $\R^N$}.
\end{aligned}\right.
\end{equation}
By Theorem 14 in \cite{S2} and Theorem \ref{thm:mr}, 
we have the following theorem for \eqref{linearshift}.
\begin{thm}\label{thm:mr2}
Let $1 < p, q < \infty$. Let $b \geq 0$.
Then, there exists a constant $\lambda_1 \geq 1$
such that the following assertion holds:
For any  
$(\bff, \bg) \in L_p((0, T), W_q^{0, 1}(\BR^N))$,
 problem  \eqref{linearshift} admits 
unique solutions  $(\bu, \bQ) \in X_{p, q, T}$
possessing the estimate 
\begin{equation}\label{mr es2-w}
\begin{aligned}
&\|<t>^b \pd_t (\bu, \bQ)\|_{L_p((0, T), W^{0, 1}_q(\BR^N))}
+
\|<t>^b (\bu, \bQ)\|_{L_p((0, T), W^{2, 3}_q(\BR^N))}\\
&\leq
C\|<t>^b(\bff, \bg)\|_{L_p((0, T), W^{0, 1}_q(\BR^N))}. 
\end{aligned}
\end{equation}
\end{thm}

\begin{proof}
Let $\bff_0$ and $\BG_0$ be the zero extension of $\bff$ and $\BG$ outside of $(0, T)$. 
Applying Laplace transform to equations \eqref{linearshift} replaced $\bff$ and $\bG$ with
$\bff_0$ and $\bG_0$ yields that
$$\begin{aligned}
(\lambda+\lambda_1)\CL[\bu] - \Delta \CL[\bu] + \nabla\CL[\fp]
+ \beta \DV\Delta \CL[\BQ] = \CL[\bff], \quad\dv \CL[\bu]=0&&\quad&
\text{in $\BR^N$}, \\
(\lambda+\lambda_1)\CL[\BQ] - \beta \bD(\CL[\bu]) - \Delta\CL[\BQ]
+a(\CL[\BQ] - \frac1N{\rm tr}\CL[\BQ]\,\bI) = \CL[\BG]
&&\quad&\text{in $\BR^N$}.
\end{aligned} $$
Let $\CA(\lambda)$ and $\CB(\lambda)$ be $\CR$-bounded solution operators given
in Theorem \ref{thm:Rbdd} and using
these operators we have
$$\bu= \CL^{-1}[\CA(\lambda+\lambda_1)(\CL[\bff_0](\lambda), \CL[\BG_0](\lambda))], \quad
\BQ = \CL^{-1}[\CB(\lambda+\lambda_1)(\CL[\bff_0](\lambda), \CL[\BG_0](\lambda))].
$$
Here, choosing $\lambda_1 > 0$ so large that $\lambda + \lambda_1 \in 
\Sigma_{\sigma, \lambda_0}$ for any $\lambda= i\tau \in i\BR$ and applying
the Weis operator valued Fourier multiplier theorem \cite{W}, we have
\begin{align*}
\|\pd_t&(\bu, \BQ)\|_{L_p(\BR, W^{0,1}_q(\BR^N))} 
+ \|(\bu, \BQ)\|_{L_p(\BR, W^{2,3}_q(\BR^N))}\\
&\leq C\|(\bff_0, \BG_0)\|_{L_p(\BR, W^{0,1}_q(\BR^N))}
= C\|(\bff, \BG)\|_{L_p((0, T), W^{0,1}_q(\BR^N))}.
\end{align*}
Noting that $|(\tau\pd_\tau)^\ell\gamma/(\lambda+\lambda_1)|
\leq C_\ell$ for any $\lambda=\gamma + i\tau \in \BR_+ + i\BR$ with
some constant $C_\ell$ depending on $\ell=0,1,2,\ldots$, and applying
the Weis operator valued Fourier multiplier theorem to 
\begin{align*}
\gamma \bu& = \CL^{-1}[\{\gamma/(\lambda+\lambda_1)\}(\lambda+\lambda_1)
\CA(\lambda+\lambda_1)(\CL[\bff_0], \CL[\BG_0])], \\
\gamma\BQ &= \CL^{-1}[\{\gamma/(\lambda+\lambda_1)\}(\lambda+\lambda_1)\CB(\lambda+\lambda_1)(\CL[\bff_0], \CL[\BG_0])].
\end{align*}
yields that 
$$\gamma\|e^{-\gamma t}(\bu, \BQ)\|_{L_p(\BR, W^{0,1}_q(\BR^N)}
\leq C\|e^{-\gamma t}(\bff_0, \BG_0)\|_{L_p( \BR, W^{0,1}_q(\BR^N))}
= C\|(\bff, \BG)\|_{L_p((0, T), W^{0,1}_q(\BR^N))}
$$
for any $\gamma>0$.  Since $\CL[(\bff_0, \BG_0)](\lambda)$ is holomorphic 
in $\BR_+ + i\BR$ as follows from $(\bff_0, \BG_0)=(0, O)$ for $t < 0$,
by Cauchy's theorem in theory of one complex variable, $\bu$ and $\BQ$
are independent of choice of $\gamma > 0$.  In particular, we have
$$\gamma\|(\bu, \BQ)\|_{L_p((-\infty, 0), W^{0,1}_q(\BR^N))} 
\leq \gamma\|e^{-\gamma t}(\bu, \BQ)\|_{L_p(\BR, W^{0,1}_q(\BR^N))} 
\leq C\|(\bff, \BG)\|_{L_p((0, T), W^{0,1}_q(\BR^N))}
$$
for any $\gamma > 0$, where $C$ is a constant independent of $\gamma$.
Thus, letting $\gamma \to \infty$ yields that $(\bu, \BQ)$ vanishes for
$t < 0$.  But, $(\bu, \BQ) \in C^0(\BR, B^{2(1-1/p)}_{q,p}(\BR^N) 
\times B^{1+2(1-1/p)}_{q,p}(\BR^N))$ as follows from real interpolation theorem, 
and so $(\bu, \BQ)|_{t=0} = (0, O)$. 

For $b \in (0, 1]$, setting $<t>^b\bu=\bv$ and $<t>^b\BQ = \BS$, we have 
$$
\left\{
\begin{aligned}
&\pd_t\bv+\lambda_1 \bv- \Delta \bv + \nabla(<t>^b\fp)
 + \beta \DV \Delta \BS =<t>^b\bff + (\pd_t<t>^b)\bu,
\quad \dv\bv=0 
& \quad&\text{in $\R^N$ for $t \in (0, T)$}, \\
&\pd_t \BS +\lambda_1 \BS - \beta \bD(\bv) 
-\Delta \BS + a\left(\BS - \frac{1}{N} \tr \,\BS \,\bI \right)
=<t>^b\BG +(\pd_t<t>^b)\BQ& \quad&\text{in $\R^N$ for $t \in (0, T)$}, \\
&(\bv, \BS) =
 (0, O) &\quad&\text{in $\R^N$}.
\end{aligned}\right.
$$
Noting that $|\pd_t(<t>^b)| \leq 1$ and applying the estimate for $(\bu, \BQ)$
yields that
\begin{align*}
\|<t>^b&\pd_t(\bu, \BQ)\|_{L_p(\BR, W^{0,1}_q(\BR^N))} 
+ \|<t>^b(\bu, \BQ)\|_{L_p(\BR, W^{2,3}_q(\BR^N))}\\
&\leq C\|<t>^b(\bff_0, \BG_0)\|_{L_p(\BR, W^{0,1}_q(\BR^N))}
= C\|<t>^b(\bff, \BG)\|_{L_p((0, T), W^{0,1}_q(\BR^N))}.
\end{align*}
If $b > 1$, the repeated use of the argument above yield 
estimates \eqref{mr es2-w} for any $b > 0$, which completes the proof of 
Theorem \ref{thm:mr2}. 
\end{proof}

We now estimate nonlinear terms $\bff(\bu, \BQ)$ and 
$\BG(\bff, \BQ)$. For notational simplicity, we write
$$\|f\|_{L_p((0, T), X(\BR^N))} = \|f\|_{L_p(X)}, 
\quad \|<t>^bf\|_{L_p((0, T), X(\BR^N))} = \|f\|_{L_{p,b}(X)},$$
where  $X = L_q$ or $X= W^\ell_q$.  
Noting \eqref{infty} and 
using Sobolev's embedding theorem provided by $q_2 > N$
and noting that $1-bp<0$ if $b=N/(2(2+\sigma))$ and $p=2+\sigma$,
we have the following estimates:
\begin{equation}\label{esnon}
\begin{aligned}
&\|\bff(\bu, \bQ)\|_{L_{p, b}(L_q)}\\
&\enskip \leq C (\|\bu\|_{L_\infty(L_q)}\|\nabla \bu\|_{L_{p, b}(W^1_{q_2})}
+ \|\bQ\|_{L_\infty(W^1_q)}\|\nabla \bQ\|_{L_{p, b}(W^2_{q_2})}
+ \|\bQ\|_{L_\infty(W^1_{q_2})}\|\nabla^3 \bQ\|_{L_{p, b}(L_q)}\\
&\enskip + \|\nabla \bQ\|_{L_\infty(L_q)}\|\nabla^2 \bQ\|_{L_{p, b}(W^1_{q_2})}),\\
&\|\bff(\bu, \bQ)\|_{L_{p, b}(L_{q_1/2})}
\leq C (\|\bu\|_{L_\infty(L_{q_1})}\|\nabla \bu\|_{L_{p, b}(L_{q_1})}
+ \|\bQ\|_{L_\infty(W^1_{q_1})}\|\nabla \bQ\|_{L_{p, b}(W^2_{q_1})}),\\
&\|\bg(\bu, \bQ)\|_{L_{p, b}(W^1_q)}\\
&\enskip \leq C (\|\nabla \bu\|_{L_{p, b}(W^1_{q_2})}\|\bQ\|_{L_\infty(W^1_q)}
+ \|\bu\|_{L_\infty(L_q)}\|\nabla \bQ\|_{L_{p, b}(W^2_{q_2})} 
+ \|\bQ\|_{L_\infty(L_q)}\|\nabla\bQ\|_{L_{p, b} (W^1_{q_2})}\\
&\enskip+ \|\bQ\|_{L_\infty(L_q)}\|\bQ\|_{L_{p, b}(W^1_{q_2})}),\\
&\|\bg(\bu, \bQ)\|_{L_{p, b}(W^1_{q_1/2})}\\
&\enskip \leq C (\|\nabla \bu\|_{L_{p, b}(W^1_{q_1})}\|\bQ\|_{L_\infty(W^1_{q_1})}
+ \|\bu\|_{L_\infty(L_{q_1})}\|\nabla^2 \bQ\|_{L_{p, b}(L_{q_1})} 
+ \|\bQ\|_{L_\infty(L_{q_1})}\|\nabla\bQ\|_{L_{p, b} (L_{q_1})}\\
&\enskip+ \|\bQ\|_{L_{\infty, b}(L_{q_1})}^2)
\end{aligned}
\end{equation}
with $q=q_1$ and $q_2$.
Therefore, by
\eqref{mr es2-w} and \eqref{esnon},
we have
\begin{equation}\label{n v1p1}
\begin{aligned}
&\sum_{q=q_1/2, q_1, q_2} \|<t>^b\pd_t (\bv_1, \bP_1)\|_{L_p((0, T), W^{0, 1}_{q}(\R^N))}+
\|<t>^b(\bv_1, \bP_1)\|_{L_p((0, T), W^{2,3}_{q}(\BR^N))} \leq  C\CN(\bu,  \bQ)(T)^2. 
\end{aligned}
\end{equation}

\subsection{Analysis of compensation equations}
In this subsection, we consider problem \eqref{compensation}.
To get estimates of $(\bv_2, \bP_2)$, we use the following estimates for 
$\{e^{\CA t}\}_{t \geq 0}$ associated with \eqref{l3}.
Recalling Theorem \ref{decay for u}, 
we know that estimates for $(\bu, \bQ)=e^{\CA t}(\bff, \bg)=(S(t)(\bff, \bg), T(t)(\bff, \bg))$
as follows:
\begin{equation}\label{decay for semi}
\begin{aligned}
&\|\nabla^j e^{\CA t} (\bff, \bg)\|_{W^{0,1}_p(\R^N)}
\leq C
t^{-\frac{N}{2} (\frac{1}{q} - \frac{1}{p})- \frac{j}{2}}
(\|(\bff, \bg)\|_{W^{0, 1}_q(\R^N)} + \|(\bff, \bg)\|_{W^{0,1}_p(\R^N)}),\\
&\|\pd_t e^{\CA t}(\bff, \bg)\|_{W^{0,1}_p(\R^N)}
\leq C
t^{-\frac{N}{2} (\frac{1}{q} - \frac{1}{p})-1}(
\|(\bff, \bg)\|_{W^{0, 1}_q(\R^N)} + \|(\bff, \bg)\|_{W^{0,1}_p(\R^N)})
\end{aligned}
\end{equation}
for $t \geq 1$, $1< q < 2 \leq p\leq \infty$, $\ell=0, 1$, 
$j=0, 1, 2$. 
Moreover, 
\begin{align}\label{conti}
\|e^{\CA t}(\bff, \bg)\|_{W^{2,3}_p(\R^N)} + 
\|\pd_t e^{\CA t}(\bff, \bg)\|_{W^{0,1}_p(\R^N)} 
\leq C\|(\bff, \bg)\|_{W^{2,3}_p(\R^N)}
\end{align}
for $0 < t < 2$.
Applying Duhamel's principle to \eqref{compensation} furnishes that
\begin{equation}\label{duhamel}
(\bv_2, \bP_2) = e^{\CA t}(\bu_0, \bQ_0)
+ \lambda_1 \int^t_0 e^{\CA (t-s)}(\bv_1, \bP_1)(\cdot, s)\, ds.
\end{equation}
Let
\begin{align*}
[[(\bv_1, \bP_1)(\cdot, s)]] 
&=\|(\bv_1, \bP_1)(\cdot, s)\|_{W^{0, 1}_{q_1/2}(\R^N)} + 
\sum_{q=q_1, q_2} \|(\bv_1, \bP_1)(\cdot, s)\|_{W^{2, 3}_q(\R^N)}.
\end{align*}
Setting 
\[
\tilde \CN(\bv_1, \bP_1)(T) = \left( \int^T_0 (<t>^b [[(\bv_1, \bP_1)(\cdot, t)]])^p \,dt\right)^{1/p}
\]
and using \eqref{mr es2-w} and \eqref{esnon}, we have
\begin{equation}\label{tilde n v1p1}
\tilde \CN(\bv_1, \bP_1)(T) \leq C \epsilon^2.
\end{equation}
In what follows, we estimate   $\CN(\bv_2, \bP_2)$
with the help of $\tilde \CN(\bv_1, \bP_1)$. 

Using \eqref{decay for semi} to obtain decay estimates of $e^{\CA t}(\bu_0, \BQ_0)$ for
$t>1$ and Theorem \ref{thm:mr} to estimate $e^{\CA t}(\bu_0, \BQ_0)$ for $0 < t < 1$
we have 
\begin{equation}\label{decay est. initial}
\CN(e^{\CA t}(\bu_0, \BQ_0)) \leq C\CI
\end{equation}
where we have set 
$$\CI = \|(\bu_0, \BQ_0)\|_{W^{0,1}_{q_1/2}(\BR^N)} + 
\sum_{q=q_1, q_2}\|(\bu_0, \BQ_0)\|_{B^{2(1-1/p)}_{q,p}(\BR^N) \times B^{1+2(1-1/p)}_{q,p}(\BR^N)}.
$$
Set $(\tilde\bv_2, \tilde\BP_2) = \int^t_0e^{\CA(t-s)}(\bv_1, \BP_1)(\cdot, s)\,ds$ and 
our main task is to estimate $(\tilde\bv_2, \tilde\BP_2)$.

\subsubsection{Estimates of spatial derivatives in $L_p$-$L_q$}

First, we consider the case $2\leq t \leq T$.
Let $(\bv_3, \bP_3)=(\bar \nabla^1 \nabla \tilde\bv_2, 
\bar \nabla ^2 \nabla \tilde\bP_2)$ when $q=q_1$
and
$(\bv_3, \bP_3)=(\bar \nabla^2 \tilde \bv_2, \bar \nabla ^3 \tilde\bP_2)$ when $q=q_2$.
Here, $\bar \nabla^m f =(\pd_x^\alpha f \mid |\alpha| \leq m)$.
\begin{align*}
&\|(\bv_3, \bP_3)(\cdot, t)\|_{L_q(\R^N)}\\
& \enskip \leq C 
\left( \int^{t/2}_0 + \int^{t-1}_{t/2} + \int^t_{t-1}\right) 
\|(\bar \nabla^1 \nabla, \bar \nabla ^2 \nabla) \text{ or }(\bar \nabla^2, \bar \nabla ^3) 
e^{\CA (t-s)} (\bv_1, \bP_1)(\cdot, s)\|_{L_q(\R^N)}\, ds\\
&\enskip =: I_q(t) + II_q(t)+III_q(t).
\end{align*}
We shall consider estimates of $I_q(t)$, $II_q(t)$, and $III_q(t)$
by \eqref{decay for semi}.
Setting $\ell = N/(2 (2 + \sigma)) + 1/2$,
we see that
all the decay rates used below, which are obtained by \eqref{decay for semi},
are greater than or equal to $\ell$.
In fact,
by \eqref{condi:pq} and
\eqref{decay for semi} with $(p, q)=(q_1, q_1/2)$, $(q_2, q_1/2)$, 
we have the following decay rates:
\begin{align*}
&\frac{N}{2}\left(\frac{2}{q_1} - \frac{1}{q_1}\right) + \frac12
= \frac{N}{2 (2 + \sigma)} + \frac{j}{2} \geq \ell \enskip (j=1, 2),\\
&\frac{N}{2}\left(\frac{2}{q_1} - \frac{1}{q_2}\right) + \frac{j}{2} 
> \frac{2N}{2 (2 + \sigma)} - \frac{N-(2+\sigma)}{2(2+\sigma)}+ \frac{j}{2} 
\geq \ell \enskip (j=0, 1, 2).
\end{align*}
Using \eqref{decay for semi} with $(p, q)=(q, q_1/2)$, we have
\begin{align*}
I_q(t)&\leq C\int^{t/2}_0 (t-s)^{-\ell}[[(\bv_1, \bP_1)(\cdot, s)]]\,ds\\
& \leq C(t/2)^{-\ell}
\left(\int^\infty_0 <s>^{-p' b}\right)^{1/p'}
\left(\int^T_0 \left(<s>^b [[(\bv_1, \bP_1)(\cdot, s)]] \right)^p \,ds\right)^{1/p}\\
&\leq Ct^{-\ell} \tilde \CN (\bv_1, \bP_1)(T).
\end{align*}
Noting that $b = N/(2(2+\sigma))$, 
we see that $\ell - b=1/2$, so that $1-(\ell - b)p<0$ for $p=2+\sigma$.
Thus, we have
\begin{equation}\label{Iq}
\begin{aligned}
\int^T_2\left(<t>^b I_q(t)\right)^p \,dt
&\leq C \int^T_2 <t>^{-(\ell - b)p}\,dt ~\tilde \CN (\bv_1, \bP_1)(T)^p\\
& \leq C \tilde \CN (\bv_1, \bP_1)(T)^p.
\end{aligned}
\end{equation}
Using $<t>^b \leq C <s>^b$ for $t/2 < s < t-1$ and H\"older's inequality, we have
\begin{equation*}\label{I}
\begin{aligned}
&<t>^b II_q(t) \\
&\enskip \leq C \left( \int^{t-1}_{t/2} (t-s)^{-\ell}\,ds \right)^{1/p'} 
\left(\int^{t-1}_{t/2}  (t-s)^{-\ell} \left(<s>^b [[(\bv_1, \bP_1)(\cdot, s)]]\right)^p\,ds\right)^{1/p}.
\end{aligned}
\end{equation*}
By Fubini's theorem, we have
\begin{equation}\label{IIq}
\begin{aligned}
\int^T_2 \left(<t>^b II_q(t)\right)^p \,dt
&\enskip \leq C \int^T_1 \int^{2s}_{s+1} (t-s)^{-\ell}\,dt 
\left(<s>^b [[(\bv_1, \bP_1)(\cdot, s)]]\right)^p\,ds\\
&\leq C \tilde \CN (\bv_1, \bP_1)(T)^p.
\end{aligned}
\end{equation}
By \eqref{conti}, we have
\[
III_q(t) \leq C \int^t_{t-1} \|(\bv_1, \bP_1)(\cdot, s)\|_{W^{2, 3}_q(\R^N)}
\leq C \int^t_{t-1} [[(\bv_1, \bP_1)(\cdot, s)]]\,ds.
\]
Employing the same method as in the estimate of $II_q(t)$, we have
\begin{equation}\label{IIIq}
\begin{aligned}
\int^T_2 \left(<t>^b III_q(t)\right)^p \,dt
\leq C \tilde \CN (\bv_1, \bP_1)(T)^p.
\end{aligned}
\end{equation}
Combining \eqref{Iq}, \eqref{IIq}, and \eqref{IIIq}, we have
\begin{equation}\label{pdx}
\int^T_2 \left(<t>^b \|(\bv_3, \bP_3)(\cdot, t)\|_{L_q(\R^N)}\right)^p \,dt
\leq C \tilde \CN (\bv_1, \bP_1)(T)^p.
\end{equation}

Next, we consider the case $0<t< \min (2, T)$.
Employing the same method as in the estimate of $III_q(t)$, we have
\begin{equation*}\label{v3p3}
\begin{aligned}
\int^{\min(T, 2)}_0 \left(<t>^b \|(\bv_3, \bP_3)(\cdot, t)\|_{L_q(\R^N)}\right)^p \,dt
\leq C \tilde \CN (\bv_1, \bP_1)(T)^p,
\end{aligned}
\end{equation*}
which combined \eqref{pdx}, we have
$$
\begin{aligned}
\int^T_0 \left(<t>^b \|(\bv_3, \bP_3)(\cdot, t)\|_{L_q(\R^N)}\right)^p \,dt
\leq C \tilde \CN (\bv_1, \bP_1)(T)^p,
\end{aligned}
$$
that is, 
\begin{equation}\label{mainest:1}
\|<t>^b\nabla(\tilde \bv_2, \tilde \BP_2)\|_{L_p((0, T), W^{1,2}_{q_1}(\BR^N)} 
+ \|<t>^b(\tilde \bv_2, \tilde \BP_2)\|_{L_p((0, T), W^{2,3}_{q_2}(\BR^N)}
\leq C\tilde \CN(\bv_1, \BP_1)(T).
\end{equation}

\subsubsection{Estimates of time derivatives in $L_p$-$L_q$}
Let $q = q_1, q_2$.
By \eqref{duhamel}, we have
\[
(\pd_t \tilde\bv_2, \pd_t \tilde\bP_2) 
= -\lambda_1 (\bv_1, \bP_1)
- \lambda_1 \int^t_0 \pd_t e^{\CA t}(\bv_1, \bP_1)(\cdot, s)\, ds.
\]
Setting $(\bv_4, \bP_4) = \int^t_0 \pd_t e^{\CA t}(\bv_1, \bP_1)(\cdot, s)\, ds$
and 
employing the same calculation as in the proof of \eqref{mainest:1}, we have
$$
\begin{aligned}
\int^T_0 \left(<t>^b \|(\bv_4, \bP_4)(\cdot, t)\|_{W^{0, 1}_q(\R^N)}\right)^p \,dt
\leq C \tilde \CN (\bv_1, \bP_1)(T)^p, 
\end{aligned}
$$
which yields that 
\begin{equation}\label{mainest:2}
\begin{aligned}
&\|<t>^b\pd_t(\tilde\bv_2, \tilde \bP_2)(\cdot, t)\|_{L_p((0, T), W^{0, 1}_q(\R^N))} 
 \leq  C\tilde \CN(\bv_1, \BP_1)(T).
\end{aligned}
\end{equation}

\subsubsection{Estimates of the lower order term in $L_\infty$-$L_q$}
Let $q = q_1$, $q_2$.
First, we consider the case $2 < t <T$.
By \eqref{duhamel}, we divide three parts as follows:
\begin{align*}
&\|(\tilde\bv_2, \tilde\bP_2)(\cdot, t)\|_{W^{0, 1}_q(\R^N)}\\
& \enskip \leq C 
\left( \int^{t/2}_0 + \int^{t-1}_{t/2} + \int^t_{t-1}\right) 
\|e^{\CA (t-s)}(\bv_1, \bP_1)(\cdot, s)\|_{W^{0, 1}_{q_1}(\R^N)}\, ds\\
&\enskip =: I_{q, 0}(t) + II_{q, 0}(t)+III_{q, 0}(t).
\end{align*}
Using \eqref{decay for semi} with $(p, q)=(q, q_1/2)$
and noting that $1-bp' < 0$ provided by $0<\sigma<1/2$
and $\frac{N}{2}\left(\frac{2}{q_1}-\frac1q\right) \geq b$ for $q=q_1, q_2$, 
we have
\begin{align}
I_{q, 0}(t)
&\leq C\int^{t/2}_0 (t-s)^{-\frac{N}{2}\left(\frac{2}{q_1}-\frac1q\right)}[[(\bv_1, \bP_1)(\cdot, s)]]\,ds\\
&\leq C(t/2)^{-b}\Bigl\{\int^{t/2}_0<s>^{-bp'}\,ds\Bigr\}\tilde \CN(\bv_1, \BP_1)(T)\\
&\leq C(t/2)^{-b} \tilde \CN (\bv_1, \bP_1)(T). \label{Iq0}\\
II_{q, 0}(t)
&\leq C \int^{t-1}_{t/2} (t-s)^{-\frac{N}{2}\left(\frac{2}{q_1}-\frac1q\right)}[[(\bv_1, \bP_1)(\cdot, s)]]\,ds \nonumber\\
&\leq C \left(\int^{t-1}_{t/2} \left((t-s)^{-b} <s>^{-b} \right)^{p'}\,ds \right)^{1/p'} 
\left( \int^{t-1}_{t/2} \left(<s>^b [[(\bv_1, \bP_1)(\cdot, s)]] \right)^p \,ds\right)^{1/p} \nonumber\\
&\leq C <t/2>^{-b} \left( \int^{t-1}_{t/2} (t-s)^{-bp'} \,ds \right)^{1/p'} 
\left(\int^{t-1}_{t/2} \left(<s>^b [[(\bv_1, \bP_1)(\cdot, s)]]\right)^p\,ds\right)^{1/p} \nonumber\\
&\leq C<t>^{-b} \tilde \CN (\bv_1, \bP_1)(T). \label{IIq0}
\end{align}
By \eqref{conti} and $<t>^b \leq C<s>^b$ for $t-1<s<t$, we have
\begin{equation}\label{IIIq0}
\begin{aligned}
III_{q, 0}(t) 
&\leq C \int^t_{t-1} [[(\bv_1, \bP_1)(\cdot, s)]]\,ds\leq  C<t>^{-b} \int^t_{t-1}<s>^b [[(\bv_1, \bP_1)(\cdot, s)]]\,ds\\
&\leq C <t>^{-b} \left( \int^t_{t-1} \,ds \right)^{1/p'} 
\left(\int^t_{t-1} \left(<s>^b [[(\bv_1, \bP_1)(\cdot, s)]]\right)^p\,ds\right)^{1/p}\\
&\leq C<t>^{-b}  \tilde \CN (\bv_1, \bP_1)(T).
\end{aligned}
\end{equation}

Combining \eqref{Iq0}, \eqref{IIq0}, and \eqref{IIIq0},
we have
\begin{equation}\label{sup}
\sup_{2 < t <T} <t>^b \|(\tilde\bv_2, \tilde\bP_2)(\cdot, t)\|_{W^{0, 1}_q(\R^N)}
\leq C \tilde \CN (\bv_1, \bP_1)(T).
\end{equation}

Next, we consider the case $0<t< \min (2, T)$.
By \eqref{conti}  we have
\begin{equation*}\label{v3p3}
\begin{aligned}
\sup_{0 < t <\min (2, T)} <t>^b \|(\tilde\bv_2, \tilde\bP_2)(\cdot, t)\|_{W^{0, 1}_q(\R^N)}
&\leq C \sup_{0 < t <\min (2, T)} <t>^b \|(\bv_1, \bP_1)(\cdot, t)\|_{W^{0, 1}_q(\R^N)}\\
&\leq C \tilde \CN (\bv_1, \bP_1)(T),
\end{aligned}
\end{equation*}
which combined \eqref{sup}, we have
\begin{equation}\label{sup2}
\begin{aligned}
\|<t>^b (\tilde \bv_2, \tilde \bP_2)\|_{L_\infty((0, T), W^{0, 1}_q(\R^N))}
\leq C \tilde \CN (\bv_1, \bP_1)(T).
\end{aligned}
\end{equation}
Therefore, by \eqref{decay est. initial}, \eqref{mainest:1}, \eqref{mainest:2},
and \eqref{sup2},
we have
\begin{equation}\label{v2p2}
\CN(\bv_2, \bP_2)(T) \leq C(\CI + \tilde\CN (\bv_1, \bP_1)(T)).
\end{equation}

\subsection{Conclusion}
 
Recall $\bv=\bv_1+\bv_2$ and $\bP=\bP_1+\bP_2$. 
We know that there exists a constant $C$
independent of $T$ such that for any $f \in L_p((0, T), W^2_q(\BR^N)) \cap W^1_p((0, T), L_q(\BR^N))$
\begin{equation}\label{time section}\begin{aligned}
\sup_{0 < t < T} \|f(\cdot, t)\|_{B^{2(1-1/p)}_{q,p}(\BR^N)} &\leq C(\|f(\cdot, 0)\|_{B^{2(1-1/p)}_{q,p}(\BR^N)} \\
&+ \|<t>^bf\|_{L_p((0, T), W^2_q(\BR^N))} + \|<t>^b\pd_tf\|_{L_p((0, T), L_q(\BR^N))})
\end{aligned}\end{equation}
as follows from the trace method of real interpolation theorem, and so using \eqref{time section}
and the fact that $(\bv_1, \BP_1)|_{t=0}=(0, O)$ we have
\begin{equation}\label{trace:est}\begin{aligned}
&\|<t>^b(\bv_1, \BP_1)\|_{L_\infty((0, T), W^{0,1}_q(\BR^N))} \\
&\quad \leq C(\|<t>^b(\bv_1, \BP_1)\|_{L_p((0, T), W^{2,3}_q(\BR^N))}
+ \|<t>^b\pd_t(\bv_1, \BP_1)\|_{L_p((0, T), W^{0,1}_q(\BR^N))}). 
\end{aligned}\end{equation}
Combining \eqref{trace:est}, \eqref{n v1p1}, and \eqref{v2p2}
yields that
\[
\CN(\bv, \bP)(T) \leq C(\CI + \CN(\bu, \BQ)(T)^2).
\]
Assuming that $\CI \leq \epsilon^2$ and recalling that  $\CN(\bu, \BQ)(T) < \epsilon$, we have
$\CN(\bv, \bP)(T) \leq C\epsilon^2$ with some constant $C$ independent of $\epsilon$.  Thus, 
choosing $\epsilon>0$ so small that $C\epsilon <1$,
we have 
\begin{equation}\label{es global}
\CN(\bv, \bP)(T) \leq \epsilon.
\end{equation}
Moreover, by the form $\bP = \bQ_0 + \int^t_0 \pd_s \bP\,ds$ and Sobolev's embedding theorem,
we have
\begin{align*}
\|\bP(\cdot, t)\|_{L_\infty(\R^N)}
& \leq C \left( \|\bQ_0\|_{W^1_{q_2}(\R^N)} 
+ \int^T_0 \|\pd_t \bP(\cdot, t)\|_{W^1_{q_2}(\R^N)}\,dt\right)\\
& \leq C\left(\|\bQ_0\|_{W^1_{q_2}(\R^N)} 
+\left( \int^\infty_0 <t>^{-p'b} \,dt\right)^{1/p'} 
\|<t>^b\pd_t \bP(\cdot, t)\|_{L_p((0, T), W^1_{q_2}(\R^N))}\right)\\
& \leq C\epsilon^2.
\end{align*}
Choosing $\epsilon>0$ so small that $C\epsilon^2 <1$ if necessary,
we have $\sup_{0<t<T}\|\bP(\cdot, t)\|_{L_\infty(\R^N)} \leq 1$.
Thus, we have $(\bv, \bP) \in \CI_{T, \epsilon}$.
Therefore, we define a map $\Phi$ acting on $(\bu, \bQ) 
\in \CI_{T, \epsilon}$ by $\Phi(\bu, \bQ) = (\bv, \bP)$, 
and then $\Phi$ is the map from $\CI_{T, \epsilon}$ into itself.
Considering the difference
$\Phi(\bu_1, \bQ_1) - \Phi(\bu_2, \bQ_2)$
for $(\bu_i, \bQ_i) \in \CI_{T, \epsilon}$ $(i = 1, 2)$,
employing the same argument as in the proof of \eqref{es global}
and choosing $\epsilon > 0$
smaller if necessary,
we see that $\Phi$ is a contraction map on $\CI_{T, \epsilon}$,
and therefore there exists a fixed point $(\bu, \bQ) \in \CI_{T, \epsilon}$
which solves \eqref{nonlinear}.
Since the existence of solutions to \eqref{nonlinear} is proved 
by the contraction mapping principle,
the uniqueness of solutions belonging to $\CI_{T, \epsilon}$ follows immediately,
which completes the proof of Theorem \ref{global}.

\end{document}